\documentclass[11pt]{article}
\usepackage{amsmath}
\usepackage{mathrsfs}
\usepackage{latexsym}
\usepackage{amssymb}
\usepackage{amsfonts}
\usepackage{pifont}
\usepackage{epsfig}
\usepackage{theorem}
\usepackage{array}

\newtheorem{theorem}{Theorem}[section]
\newtheorem{proposition}{Proposition}[section]
\newtheorem{lemma}{Lemma}[section]
\newtheorem{corollary}{Corollary}[section]

\newtheorem{definition}{Definition}[section]

{\theorembodyfont{\rmfamily}

\newtheorem{remark}{Remark}
}

\numberwithin{equation}{section}

\def\qed{\hfill{$\Box$} \\}

\def\expect{{\mathbb  E}}
\def\Pr{{\mathbb P}}

\def\real{{\mathbb  R}}

\def\eqdef{\triangleq}

\def\ind{{\bf 1}}

\renewcommand{\leq}{\leqslant}
\renewcommand{\geq}{\geqslant}

\def\U{U}
\def\A{A}

\begin{document}


\title{\textbf{Characterizing Heavy-Tailed Distributions Induced by Retransmissions}
 }
\author{Predrag R. Jelenkovi\'c \hspace{0.5in} Jian Tan  \\
\begin{tabular} {c}
\small Department of Electrical Engineering \\
\small Columbia University, New York, NY 10027 \\
\small \{predrag, jiantan\}@ee.columbia.edu \\
\end{tabular}
 }
\date{September 7, 2007}
\maketitle

\begin{abstract}
\hspace{-1.5mm}\footnote{\hspace{-2mm} { \small This work was
supported in part by the NSF under grant number CNS-0435168.}}
\hspace{-1.5mm}\footnote{\hspace{-2mm} { \small Technical Report
EE2007-09-07, Department of Electrical Engineering, Columbia
University,  \\\phantom{XX} New York, NY,  September 7, 2007.}}
 \noindent Consider a generic data
unit of random size $L$ that needs to be transmitted over a channel
of unit capacity.  The channel availability dynamics is modeled as
an i.i.d. sequence $\{A, A_i\}_{i\geq 1}$ that is independent of
$L$. During each period of time that the channel becomes available,
say $A_i$, we attempt to transmit the data unit. If $L\le A_i$, the
transmission is considered successful; otherwise, we wait for the
next available period $A_{i+1}$ and attempt to retransmit the data
from the beginning. We investigate the asymptotic properties of the
number of retransmissions $N$ and the total transmission time $T$
until the data is successfully transmitted. In the context of
studying the completion times in systems with failures where jobs
restart from the beginning, it was first recognized in
\cite{FS05,SL06} that this model results in power law and, in
general, heavy-tailed delays. The main objective of this paper is to
uncover the detailed structure of this class of heavy-tailed
distributions induced by retransmissions.

\noindent More precisely, we study how the functional dependence
$(\Pr[L>x])^{-1} \approx \Phi((\Pr[A>x])^{-1})$ impacts the
distributions of $N$ and $T$; the approximation $\approx$ will be
appropriately defined in the paper depending on the context. In the
functional space of $\Phi(\cdot)$, we discover several functional
criticality points that separate classes of different functional
behavior of the distribution of $N$. For example, we show that if
$\log (\Phi(n))$ is slowly varying, then $\log (\Pr[N>n])$ is
essentially slowly varying as well. Interestingly, if $\log
(\Phi(n))$ grows slower than $e^{\sqrt{\log n}}$ then we have the
asymptotic equivalence $\log (\Pr[N>n]) \approx -\log(\Phi(n))$.
However, if $\log (\Phi(n))$ grows faster than $e^{\sqrt{\log n}}$ ,
this asymptotic equivalence does not hold and admits a different
functional form. Similarly, different types of functional behavior
are shown for moderately heavy tails (Weibull distributions) where
$\log (\Pr[N>n]) \approx -(\log \Phi(n))^{1/(\beta+1)}$ assuming
$\log \Phi(n) \approx n^{\beta}$, as well as the nearly exponential
ones of the form $\log (\Pr[N>n]) \approx -n/(\log n)^{1/\gamma},
\gamma>0$ when $\Phi(\cdot)$ grows faster than two exponential
scales $\log \log \left(\Phi(n)\right) \approx n^{\gamma}$.

\noindent We also discuss the engineering implications of our
results on communication networks since retransmission strategy is a
fundamental component of the existing network protocols on all
communication layers, from the physical to the application one.

\vspace{0.5cm}

\noindent \textbf{Keywords}: \noindent Retransmissions, Channel
(systems) with failures, Restarts, Origins of heavy-tails
(subexponentiality), Gaussian distributions, Exponential
distributions, Weibull distributions, Log-normal distributions,
Power laws.
\end{abstract}

\newpage


%

\section{Introduction}\label{s:intro}
Retransmissions represent one of the most fundamental approaches in
communication networks that guarantee data delivery in the presence
of channel failures. These types of mechanisms have been employed on
all networking layers, including, for example, Automatic Repeat
reQuest (ARQ) protocol (e.g., see Section~2.4 of \cite{BG92}) in the
data link layer where
  a packet is resent automatically in case of an error;  contention based ALOHA type
  protocols in the medium access control (MAC) layer that use random backoff and retransmission mechanism to
  recover data from collisions; end-to-end acknowledgement for  multi-hop transmissions in the transport layer; HTTP downloading scheme in
  the application layer, etc. We discuss the engineering
  implications of our results at the end of this introduction and,
  in  more detail, in Section \ref{s:APP}.

As briefly stated in the abstract, we use the following generic
channel with failures \cite{PJ07RETRANS} to model the preceding
situations. The channel dynamics is described as an on-off process
$\{(A, U), (A_i,U_i)\}_{i\geq 1}$ with alternating  periods when
channel is available $A_i$ and unavailable $U_i$, respectively;
$(A,A_i)_{i\geq 1}$ and $(U, U_i)_{i\geq 1}$ are two independent
sequences of i.i.d random variables. In each period of time that the
channel becomes available, say $A_i$, we attempt to transmit the
data unit of random size $L$. If $L\le A_i$, we say that the
transmission is successful; otherwise, we wait for the next period
$A_{i+1}$ when the channel is available and attempt to retransmit
the data from the beginning. We study the asymptotic properties of
the distributions of the total transmission time $T$ and number of
retransmissions $N$, for the precise definitions of these variables
and the model, see the following Subsection~\ref{s:model}.

The preceding model was introduced and studied in \cite{KN87} and,
apart from the already mentioned applications in communications, it
represents a generic model for other situations where jobs have to
restart from the beginning after a failure. It was first recognized
in \cite{FS05} that this model results in power law distributions
when the distributions of $L$ and $A$ have a matrix exponential
representation, and this result was rigorously proved and further
generalized in \cite{SL06}. Under  more general conditions,
\cite{PJ07RETRANS} discovers that the distributions of $N$ and $T$
follow power laws with the same exponent $\alpha$ as long as $\log
\Pr[L>x]\approx \alpha \log\Pr[A>x]$ for large $x$, which implies
that power law distributions, possibly with infinite mean
($0<\alpha<1$) and variance ($0<\alpha<2$),  may arise even when
transmitting superexponential (e.g., Gaussian) documents/packets.
More recent results on the heavy-tailed completion times in a system
with failures are developed in \cite{Asmussen07}. In this paper, we
further characterize this class of heavy-tailed distributions that
are induced by retransmissions.

Technically speaking, our proofs are based on the method introduced
in \cite{PJ07RETRANS} that uses the following key arguments. First,
in exploring the distribution of $N$, we assume that the functional
relationship $\Phi(\cdot)$, with $\bar{F}^{-1}(x) \approx
\Phi(\bar{G}^{-1}(x))$ between the probability distributions of
$\bar{F}(x)\eqdef \Pr[L>x]$ and $\bar{G}(x)\eqdef \Pr[A>x]$, is
eventually monotonically increasing, which guarantees the existence
of an asymptotic inverse $\Phi^{{\tiny \leftarrow}}(\cdot)$ of
$\Phi(\cdot)$, and then, we use the result that $\bar{F}(L)$ is a
uniform random variable on $(0,1)$ given that $\bar{F}(\cdot)$ is
continuous (see \cite{PJ07RETRANS,Jelen07ALOHA}), e.g., for
$\bar{F}(x)=\left(\bar{G}(x) \right)^{\alpha}, \alpha>0$, the key
argument on the uniform distribution of $\bar{F}(L)$ from
\cite{PJ07RETRANS} can be illustrated as
$$\Pr[N>n]=\expect\left[(1-\bar{G}(L))^n
\right]\approx
\expect[e^{-n\bar{G}(L)}]=\expect\left[e^{-n\bar{F}^{1/\alpha}(L)}\right]=
\frac{\Gamma(\alpha+1)}{n^{\alpha}}.$$ Second, in contrast to
\cite{SL06,Asmussen07}, instead of studying the total transmission
time $T$ directly, we study a simpler quantity $N$ and then use the
large deviations technique to investigate $T$, since $T$ can be
represented as a sum of $L$ and $\{(A_i+U_i)\}_{1\leq i < N}$; see
equation (\ref{defeq:T}) in the next subsection. Hence, our analysis
is entirely probabilistic, which differs from the work in
\cite{Asmussen07} that relies on Tauberian theorems.

More precisely, we extend the results from
\cite{Asmussen07,PJ07RETRANS} under a more
 unified framework and study how the functional dependence
 between the data characteristics and
channel  dynamics in the form $(\Pr[L>x])^{-1} \approx
\Phi(\Pr[A>x])^{-1})$ impacts the distribution of $N$,  where the
approximation $\approx$ will be possibly differently defined
according to the context. In the functional space of $\Phi(n)$, we
identify several functional criticality points that define different
classes of functional behavior of the distribution of $N$.
Specifically, in Subsection \ref{ss:VHA}, we show that if $\Phi(n)$
is dominantly varying, e.g., regularly varying, then $\Pr[N>n]
\approx \Phi(n)^{-1}$; see Proposition \ref{eq:proposition1} and
Theorem \ref{theorem:asympN2}.  As shown in Proposition
\ref{eq:proposition2}, the preceding tail equivalence between
$\Pr[N>n]$ and $\Phi(n)^{-1}$ basically does not hold if $\Phi(x)$
is not dominantly varying, e.g., if $\Phi(x)$ is lognormal.
Furthermore, we show in a weaker form that if $\log (\Phi(n))$ is
slowly varying, then $\log \left((\Pr[N>n])^{-1} \right)$ is
essentially slowly varying as well, as proved in Proposition
\ref{prop:slowVarying}.  Interestingly, if $\log (\Phi(n))$ grows
slower than $e^{\sqrt{\log n}}$ then we have the asymptotic
equivalence $\log \left(\Pr[N>n]\right) \approx -\log(\Phi(n))$ as
shown in Theorem \ref{theorem:1I} and Corollary \ref{corollary:1},
which implies parts (1:1), (2:1) and (2:2) of Theorem 2.1 in
\cite{Asmussen07} and extends Theorem 2 in \cite{PJ07RETRANS}.
However, if $\log (\Phi(n))$ grows faster than $e^{\sqrt{\log n}}$,
this asymptotic equivalence does not hold and we demonstrate a
different functional form in Proposition \ref{prop:betweenHalfOne}.

Next, for lighter distributions of Weibull type, in Subsection
\ref{ss:weibull}, we show that if $\log (\Phi(n))$ is regularly
varying with index $\beta>0$, then basically one obtains Weibull
distribution for $N$, i.e., $\log \left(\Pr[N>n] \right) \approx
-\left( \log \Phi(n) \right)^{1/(\beta+1)}$, as shown in Theorem
\ref{theorem:2I}, which we term moderately heavy (Weibull tail)
asymptotics; this result implies part (1:2) of Theorem 2.1 in
\cite{Asmussen07}, and provides a more precise logarithmic
asymptotics instead of a double logarithmic limit. Finally, in
Subsection \ref{ss:nearly}, we consider the situation when the
separation between $\Pr[L>x]$ and $\Pr[A>x]$ is very large, i.e.,
their distributions are roughly separated by more than two
exponential scales ($\log\log \left( \Phi(n)\right) \approx
n^{\gamma}$). This separation results in what we call the nearly
exponential distribution for $N$ in the form $\log
\left(\Pr[N>n]\right) \approx -n/(\log n)^{1/\gamma}$.

After the preceding characterization of the different classes of
distributional behavior for $N$, we study in Subsection
\ref{ss:asmpT} the total transmission time $T$. As previously stated
for studying $T$, we use the large deviation results  since $T$ can
be represented as the sum of $L$ and $\{(A_i+U_i\}_{1\leq i < N}$.
In this context, our primary results show that: (i) when
$\Phi(\cdot)$ is regularly varying, we derive the exact asymptotics
for $T$ in Theorem \ref{theorem:asympT}. (ii) when $\log
(\Phi(\cdot))$ is slowly varying, we obtain the logarithmic
asymptotics for $T$ in Theorem \ref{theorem:T}. (iii) when $\log
(\Phi(\cdot))$ is regularly varying with positive index, we derive,
in a different scale than in Theorem \ref{theorem:T}, the
logarithmic asymptotics in Theorem \ref{theorem:TWeibull}. Note that
the preceding three results on $T$ correspond to Theorems
\ref{theorem:asympN2} i), \ref{theorem:1I} and \ref{theorem:2I} on
$N$, respectively. Similarly, one can derive the respective
statements on $\Pr[T>t]$ for other results on $\Pr[N>n]$, but we
omit this to avoid lengthy expositions and repetitions.
Interestingly, we want to point out that, unlike Theorems
\ref{theorem:asympT} and \ref{theorem:T} requiring no conditions on
$A$ (Theorem \ref{theorem:asympT} needs $\expect[A]<\infty$), the
minimum conditions needed for Theorem \ref{theorem:TWeibull}, as
shown by Proposition \ref{prop:WeibullBalance},  basically involve a
balance between  the tail decays of $\Pr[A>x]$ and $\Pr[L>x]$.

From a practical perspective, our results suggest that careful
examination and possible redesign of retransmission based protocols
in communication networks might be needed. This is especially the
case for Ad  Hoc and resource limited sensor networks, where
frequent channel failures occur due to a variety of reasons,
including  signal fading,  multipath effects, interference,
contention with other nodes, obstructions, node mobility, and other
changes in the environment \cite{Ra02}.  In engineering
applications, our main discovery is the matching between the
statistical characteristics of the channel and transmitted data
(packets).   On the network application layer, most of us have been
inconvenienced when the connections would brake while we are
downloading a large file from the Internet. This issue has been
already recognized in practice where software for downloading files
was developed that would save the intermediate data (checkpoints)
and resume the download from the point when the connection was
broken. However, our results emphasize that, in the presence of
frequently failing connections, the long delays may arise even when
downloading relatively small documents. Hence, we argue that one
might need to modify the application layer software, especially  for
the wireless environment,  by introducing checkpoints even for small
to moderate size documents. In our related papers, we found that
several well-known retransmission based protocols in different
layers of networking architecture  can lead to power law delays,
e.g., ALOHA type protocols in MAC layer \cite{Jelen07ALOHA} and
end-to-end acknowledgements in transport layer \cite{Jelen07e2e}.
These new findings suggest that special care should be taken when
designing robust networking protocols, especially in the wireless
environment where channel failures are frequent. We discuss these
and other engineering implications of our results in
Section~\ref{s:APP}.

We also discuss possible solutions to alleviate this problem, such
as  assigning checkpoints, breaking large packets into smaller units
preferably by using dynamic packet fragmentation techniques
\cite{Jelen07frag}. Clearly,  there is a tradeoff between the sizes
of these newly created packets and the throughput since, if the
packets are too small, they will mostly contain the packet headers
and, thus, very little useful information.

Finally, we would like to point out that, in addition to the
preceding applications in communication networks
\cite{Jelen07e2e,PJ07RETRANS,Jelen07ALOHA,Jelen07frag} and job
processing on machines with failures \cite{FS05,SL06}, the model
studied in this paper may represent a basis for understanding more
complex failure prone systems, e.g., see the recent study on
parallel computing in \cite{Asmussen07July}.

The rest of the paper is organized as follows. After a detailed
description of the channel model in the next Subsection
\ref{s:model}, we present our main results in Section
\ref{s:results} that is composed of two parts: the asymptotics of
the distribution of $N$ in Subsection \ref{ss:asmpN} and the
asymptotics of the distribution of $T$ in Subsection \ref{ss:asmpT}.
In Subsection \ref{ss:asmpN} we study three types of distinct
behavior, i.e., the very heavy asymptotics in Subsection
\ref{ss:VHA}, the medium heavy (Weibull) asymptotics in Subsection
\ref{ss:weibull} and the nearly exponential asymptotics in
Subsection \ref{ss:nearly}. Then, we concludes the paper with
engineering implications in Section \ref{s:APP}, which is followed
by Section \ref{s:proofs} that contains some of the technical proofs
that have been deferred from the preceding sections.

\subsection{Description of the Channel}\label{s:model}
 In this section, we formally describe our model and provide necessary definitions and notation.
Consider transmitting a generic data unit of  random size $L$ over a
channel with failures. Without loss of generality, we assume that
the channel is of unit capacity. The channel dynamics is modeled as
an on-off process $\{(A_i,U_i)\}_{i\geq 1}$ with alternating
independent periods when channel is available $A_i$ and unavailable
$U_i$, respectively. In each period of time that the channel becomes
available, say $A_i$, we attempt to transmit the data unit and, if
$L\le A_i$, we say that the transmission was successful; otherwise,
we wait for the next period $A_{i+1}$ when the channel is available
and attempt to retransmit the data from the beginning.
  A sketch of the
  model depicting the system is drawn in Figure \ref{fig:tc}.
 \begin{figure}[h]
\centering
\includegraphics[width=13cm]{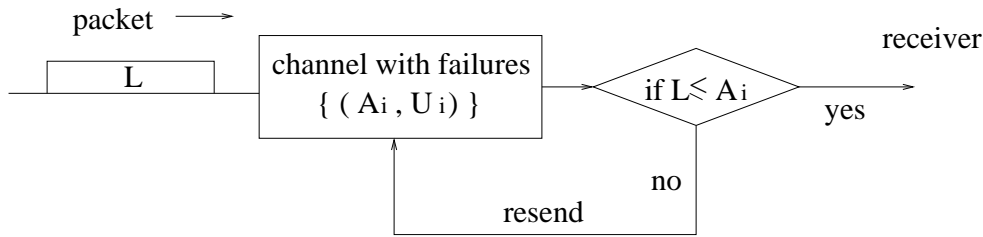}
  \caption{Packets sent over a channel with failures}
\label{fig:tc}
\end{figure}

  Assume that $\{ U, U_i \}_{i \geq 1}$ and $\{ A, A_i \}_{i\geq 1}$
 are two mutually independent sequences of
 i.i.d. random variables.
 \begin{definition}\label{def:NT}
 The total number of (re)transmissions for a generic data unit of length $L$ is defined
 as
 \begin{equation*}
  N \eqdef  \inf\{ n: {\A}_n\geq L \},
  \end{equation*}
  and, the total transmission time for the data unit is defined
 as
 \begin{equation}\label{defeq:T}
 T \eqdef \sum_{i=1}^{N-1}( {\A}_i +{\U}_i)+L.
 \end{equation}
 \end{definition}
We  denote
 the complementary cumulative distribution functions for $A$ and
 $L$, respectively, as
  \begin{equation*}
  \bar{G}(x)\eqdef \Pr[{\A}>x]
  \end{equation*} and
  \begin{equation*}
  \bar{F}(x)\eqdef \Pr[L>x].
  \end{equation*}

It was first discovered in Theorem~6 of
 \cite{SL06} that this model leads to subexponential delay $T$ under quite general conditions. The following
 slightly more general proposition was proven in Lemma 1 of \cite{PJ07RETRANS} using probabilistic arguments
  (see also Proposition 1.2 in \cite{Asmussen07}).
 \begin{proposition}\label{prop:sub}
  If $\bar{F}(x)>0$ for all $x\geq 0$, then both $N$ and $T$ are
  \emph{subexponential} in the following sense that,   for any $\epsilon>0$,
  \begin{equation}\label{eq:subN}
        e^{\epsilon n}\Pr[N>n] \rightarrow \infty \;\; \text{as}\; n
        \rightarrow \infty
  \end{equation}
  and
   \begin{equation}\label{eq:subT}
        e^{\epsilon t}\Pr[T>t] \rightarrow \infty \;\; \text{as}\; t
        \rightarrow \infty.
  \end{equation}
 \end{proposition}

Clearly,  the preceding proposition defines a class of
subexponential distributions that are induced by retransmissions;
the \textbf{proof} of this proposition is presented in Subsection
\ref{ss:sub} for readers' convenience. The main study of this paper
is to uncover the detailed structure of this class of distributions.
More precisely, we investigate how the functional dependence of
$\bar{F}$ and $\bar{G}$ (stated in the form $\bar{F}^{-1}(x) \approx
\Phi( \bar{G}^{-1}(x) )$) impacts the tail behavior of the
distributions of both $N$ and $T$, and the exact meaning of
$\approx$ will be defined according to the context.


\section{Main Results}\label{s:results}
This section presents our main results. Here, we assume that
$\bar{F}(x)$ is a continuous function with support on $[0,\infty)$.
If $\bar{F}(x)$ is lattice valued, our results may still hold; see
Remarks~\ref{remark:lattice}~and~\ref{remark:lattice2}. If
$\bar{F}(x)$ has only a finite support,  we discuss this situation
in Section \ref{s:APP}; see also Example $3$ in Section IV of
\cite{PJ07RETRANS} and Section 3 of \cite{Asmussen07}. According to
(\ref{defeq:T}), the total transmission time $T$ naturally depends
on the number of transmissions $N$,  and therefore,  we first study
the distributional properties of the number of transmissions $N$ in
Subsection \ref{ss:asmpN}, and then evaluate the total transmission
time $T$ using the large deviation approach in Subsection
\ref{ss:asmpT}.

\subsection{Asymptotics of the Distribution of the Number of Retransmissions \textbf{$N$}}\label{ss:asmpN}
This subsection presents the asymptotic results for the number of
retransmissions $N$ depending on the functional relationship
$\Phi(\cdot)$ between $\bar{F}$ and $\bar{G}$. Informally, we study
three scenarios: very heavy asymptotics (when $\log(\Phi(n))$ is
slowly varying), medium heavy (Weibull) asymptotics (when
$\log(\Phi(n))$ is regularly varying), and nearly exponential (when
$\log \log ( \Phi(n))$ is regularly varying), where within and
between these subclasses we also identify critical functional points
that define different distributional behavior of $N$.

More precisely, we show that:
\begin{enumerate}
\item If $\Phi(n)$ is dominantly regularly varying, e.g., regularly varying, then
$\Pr[N>n] \approx \Phi(n)^{-1}$, as stated in Proposition
\ref{eq:proposition1} and Theorem \ref{theorem:asympN2}.
\item If $\Phi(n)$ is not dominantly regularly varying, e.g.,  $\Phi(n)^{-1}$ being
lognormal, the preceding tail equivalence $\Pr[N>n] \approx
\Phi(n)^{-1}$ basically does not hold, as shown in Proposition
\ref{eq:proposition2}.  However, we show in a weaker form that,  if
$\log (\Phi(n))$ is slowly varying, then $\log (\Pr[N>n])$ is
essentially slowly varying as well, as proved in Proposition
\ref{prop:slowVarying}. Interestingly, within this class, we
discover two types of distinct functional behavior of $\log
\Pr[N>n]$ depending on the growth of $\log (\Phi(n))$:
 \begin{enumerate}
\item If $\log (\Phi(n))$ grows slower than $e^{\sqrt{\log n}}$, then we have the
asymptotic equivalence $ \log (\Pr[N>n]) \approx -\log(\Phi(n)) $,
as shown in Theorem \ref{theorem:1I}.
\item This asymptotic equivalence does not hold if $\log (\Phi(n))$ grows faster
than $e^{\sqrt{\log n}}$, and we demonstrate a different functional
form in Proposition \ref{prop:betweenHalfOne}.
\end{enumerate}
\item If $\log (\Phi(n))$ is regularly varying with index $\beta>0$, then
basically one obtains a Weibull distribution for $N$, $\log
\Pr[N>n]\approx - (\log \Phi(n))^{1/(\beta+1)}$,  as presented in
Theorem \ref{theorem:2I}.
\item When the decay of $\Pr[L>x]$ is much faster than $\Pr[A>x]$, i.e., $\log \log \Pr[L>x]^{-1} \approx R_{\gamma}(\log
      \Pr[A>x]^{-1})$ with $R_{\gamma}(\cdot), \gamma>1$ being regularly varying, we
obtain nearly exponential distributions for $N$ in the form $\log
(\Pr[N>n]) \approx n/R_{\gamma}^{\leftarrow}(\log n)$ with
$R_{\gamma}^{\leftarrow}(n)$ being regularly varying with
$0<1/\gamma<1$, implying that $R_{\gamma}^{\leftarrow}(\log n)$ is
slowly varying; see Theorem \ref{theorem:nearly}.
\end{enumerate}

Our \textbf{proving method} is based on the following two key
arguments:
\begin{enumerate}
\item $\Phi(x)$ is eventually monotonically increasing, which
guarantees the existence of an inverse function
$\Phi^{{\leftarrow}}(x)$ of $\Phi(x)$ when $x$ is large enough.
\item $\bar{F}(x)$ is continuous, which implies that $V \eqdef \bar{F}(L)$ is a
uniform random variable on $(0,1)$, e.g., see Proposition 2.1 in
Chapter 10 of \cite{Ross02}. Furthermore, our method essentially
extends to lattice valued $\bar{F}(x)$ as well, as discussed in
Remarks~\ref{remark:lattice}~and~\ref{remark:lattice2}.
\end{enumerate}

\subsubsection{Very Heavy Asymptotics}\label{ss:VHA}
This subsection studies the situation when the distribution of the
number of retransmissions $N$ is heavier than Weibull distributions.
Specifically,  we answer under what conditions  $\Pr[N>n] \approx
\Phi(n)^{-1}$ holds assuming $\bar{F}^{-1}(x) \approx \Phi(
\bar{G}^{-1}(x) )$, meaning that the complementary cumulative
distribution function of $N$ is of the same form (in terms of
$\Phi(\cdot)$) as the functional relationship $\Phi(\cdot)$ between
$\bar{F}$ and $\bar{G}$.

We term this subclass very heavy distributions since if $\log
(\Phi(\cdot))$ is slowly varying, then the number of retransmissions
$N$ is always heavier than Weibull distribution, which is stated in
the following Proposition \ref{prop:slowVarying}.
\begin{proposition}\label{prop:slowVarying}
If $\log (\Phi(\cdot))$ is slowly varying and
\begin{equation}\label{eq:conditionI2}
   \lim_{x\rightarrow \infty}\frac{\log \left(\bar{F}(x)^{-1}\right)}{\log \left( \Phi \left( \bar{G}(x)^{-1} \right) \right)}=1,
   \end{equation}
   then, for any $\epsilon>0$,
as $n \to \infty$,
\begin{equation}
\lim_{n \to \infty}\frac{\log
\left(\Pr[N>n]^{-1}\right)}{n^{\epsilon}} = 0. \nonumber
\end{equation}
\end{proposition}
The \textbf{proof} of this proposition will be presented in
Subsection \ref{ss:slowVarying}. In the remainder of this subsection
we study the detailed structure of this class of  distributions that
have very heavy tails. The Weibull distribution will be studied in
the next Subsection \ref{ss:weibull} on medium heavy asymptotics.

\begin{definition}
For  an eventually non-decreasing function $\Phi(x): \real ^+ \to
\real ^+$, we say that $\Phi(x)$ is dominantly regularly varying if
\begin{equation}\label{eq:dominantRV}
     \varlimsup_{x \to \infty} \frac{\Phi(e x)}{ \Phi(x)}  < \infty,
\end{equation}
where $e \equiv \exp (1)$.
\end{definition}

In the paper we use the following standard notation. For any two
real functions $a(t)$ and $b(t)$ and fixed $t_{0}\in
\real\cup\{\infty\}$, we use $a(t)\thicksim b(t)$ as $t\rightarrow
t_{0}$ to denote $\lim_{t\rightarrow t_{0}} [a(t)/b(t)]=1$.
Similarly, we say that $a(t)\gtrsim b(t)$ as $t\rightarrow t_{0}$ if
$\varliminf_{t\rightarrow t_{0}}a(t)/b(t)\ge 1$; $a(t)\lesssim b(t)$
has a complementary definition. In addition, we say that
$a(t)=o(b(t))$ as $t\rightarrow t_0$ if $\lim_{t\rightarrow t_0}
a(t)/b(t)= 0$. When $t_0=\infty$, we often simply write
$a(t)=o(b(t))$ without explicitly stating $t\rightarrow \infty$ in
order to simplify the notation. Also, we use the standard definition
of an inverse function $f^{\leftarrow}(x) \eqdef \inf \{ y: f(y)> x
\}$ for a non-decreasing function $f(x)$; note that the notation
$f^{-1}(x)$ is reserved for $1/f(x)$.

The following two propositions show that $\Pr[N>n]$ is tail
equivalent to $\Phi(n)^{-1}$ basically only when $\Phi(n)$ is
dominantly regularly varying.
\begin{proposition}\label{eq:proposition1}
If,  as $x\to \infty$,
\begin{equation}\label{eq:propo1A}
 \bar{F}^{-1}(x) \sim \Phi( \bar{G}^{-1}(x) ),
 \end{equation}
 then, there is finite $c\geq 1$ such that
    \begin{equation}
        c^{-1} \leq  \varliminf_{n \to \infty } \Pr[N>n]\Phi(n) \leq  \varlimsup_{n \to
        \infty}\Pr[N>n]\Phi(n) \leq c. \nonumber
    \end{equation}
\end{proposition}

\begin{remark}
Note that for this result as well as those in the rest of the paper
we could have equivalently assumed  that $\bar{F}(x) \sim
\Phi(\bar{G}(x))$ where $\Phi(\cdot)$ is eventually non-increasing
and satisfies the appropriate regularity conditions in the
neighborhood of $0$, e.g., condition (\ref{eq:propo1A}) would be
restated in the neighborhood of $0$. In this case, the respective
statement would be in the form $\Pr[N>n] \approx \Phi(n^{-1})$.
Furthermore, the current form has additional notational benefits in
the later sections, e.g., $\log \log \Phi(n)$ would need to be
replaced by $\log\left( - \log (\Phi(n^{-1})) \right)$ in (say)
Proposition \ref{prop:betweenHalfOne}.
\end{remark}

\begin{proposition}\label{eq:proposition2}
If (\ref{eq:propo1A}) is satisfied and $\Phi(x)$ is eventually
non-decreasing with
\begin{equation}
     \lim_{x \to \infty} \frac{\Phi(e x)}{ \Phi(x)}= \infty,
     \nonumber
\end{equation}
then,
    \begin{equation}
        \lim_{n \to \infty }\Pr[N>n]\Phi(n) = \infty.
        \nonumber
    \end{equation}
\end{proposition}

When $\Phi(\cdot)$ is regularly varying, which is a subset of the
dominantly regularly varying functions, we can compute the exact
asymptotics of the distribution of $N$.
\begin{theorem}\label{theorem:asympN2}
 Assuming  $\bar{F}^{-1}(x) \thicksim \Phi\left(
\bar{G}^{-1}(x)\right)$ where $\Phi(\cdot)$ is regularly varying
with index $\alpha$, we obtain:

i) If  $\alpha>0$, then, as $n \rightarrow \infty$,
 \begin{equation}\label{eq:asympNa}
\Pr[N>n]\thicksim\frac{\Gamma( \alpha+1 )}{ \Phi\left( n \right) }.
\end{equation}

ii) If $\alpha=0$ (meaning $\Phi(\cdot)$ is  slowly varying) and
$\Phi(x)$ is eventually non-decreasing,
then, as $n \rightarrow \infty$,
 \begin{equation}\label{eq:asympNb}
\Pr[N>n]\thicksim \frac{1}{\Phi(n)}.
\end{equation}
\end{theorem}
\begin{remark}
For $\alpha>0$, this theorem was proved in Theorem 4 of
\cite{PJ07RETRANS} using the method that we further expand in this
paper; alternatively,  a similar result for $T$ was proved using
Tauberian method in Theorem 2.2 of \cite{Asmussen07}. We will prove
the corresponding result for $T$ in Theorem \ref{theorem:asympT} in
Subsection \ref{ss:asmpT}.
\end{remark}

\begin{remark}[Lattice variables]\label{remark:lattice}
Note that if $\bar{F}(x)$ and $\bar{G}(x)$ are {\it lattice valued},
then the distribution of $N$ may still be tail equivalent to
$\Phi(n)^{-1}$, as in Proposition \ref{eq:proposition2}, but the
constant in front of $\Phi(n)^{-1}$ may be different from
$\Gamma(\alpha+1)$, e.g., if $\Pr[L>n] \sim e^{-p n}, p>0$ and
$\Pr[A>n] \sim e^{-q n}, q>0$, then this constant is between $
e^{-p} \Gamma(1+p/q)$ and $e^{p}\Gamma(1+p/)$.
\end{remark}

Before moving to the proof,  we state two straightforward
consequences of the preceding theorems; see also Theorem 1 and
Corollary 1 in \cite{PJ07RETRANS}. The following corollary allows
$\bar{F}$ and $\bar{G}$ to have exponential type distributions, and
the corresponding result for $T$ was first derived in Theorem 7 of
\cite{SL06}.
\begin{corollary}
 Assume that $\bar{G}(x)\thicksim e^{-\beta x}$ and $\bar{F}(x)\thicksim
a x^be^{-\delta x}$ where $b\in \real$ and $a, \beta>0$,  then,
 \begin{equation}
      \Pr[N>n]\thicksim
      a\Gamma\left(\frac{\delta}{\beta}+1\right){\beta}^{-b}\frac{(\log
      t)^b}{t^{\frac{\delta}{\beta}}}.
 \end{equation}
\end{corollary}
\begin{proof}
 It is easy to verify that, as $x\rightarrow \infty$,
\begin{equation*}
 \bar{F}^{-1}(x)
\thicksim a^{-1} \beta^{b} \left( \log \bar{G}^{-1}(x) \right)^{-b}
\bar{G}(x)^{-\frac{\delta}{\beta}},
\end{equation*}
 and, therefore, we can choose
 \begin{equation*}
        \Phi(x)=a^{-1} \beta^{b} \left( \log x \right)^{-b}
x^{\frac{\delta}{\beta}},
 \end{equation*}
 which, by using Theorem \ref{theorem:asympT}, finishes the
 proof.\qed
\end{proof}

The following corollary  allows $\bar{F}$ and $\bar{G}$ to have
normal-like distributions, i.e., much lighter tails than exponential
distributions, as shown in Corollary 1 of \cite{PJ07RETRANS} (see
also Corollary 2.2 in \cite{Asmussen07}).
\begin{corollary}\label{co:normal}
Suppose $\bar{G}(x)=\Pr[|N(0,{\sigma}_{A}^2)|>x]$ and
$\bar{F}(x)=\Pr[|N(0,{\sigma}_{L}^2)|>x]$, where $N(0,{\sigma}^2)$
is a Gaussian random variable with mean zero and variance
${\sigma}^2$, then,
\begin{equation}\label{eq:normal}
    \Pr[N>n] \thicksim  \Gamma\left( \alpha+1 \right)\alpha^{-1/2}
    \frac{\left( \pi \log n \right)^{\frac{1}{2}\left( \alpha-1 \right) }}
     { n^{\alpha} },
\end{equation}
where $\alpha = {\sigma_A}^2/{\sigma_L}^2$.
\end{corollary}
\begin{proof}
   First, notice that
   $$\Pr[|N(0,{\sigma}^2)|>x]\thicksim
   \frac{2\sigma}{\sqrt{2\pi}x}e^{-\frac{x^2}{2{\sigma}^2}},
   $$
   and therefore, recalling  $\alpha = {\sigma_A}^2/{\sigma_L}^2$, we obtain
   $$
     \bar{F}(x) \thicksim {\pi}^{\frac{1}{2}\left( \alpha-1 \right)}
      {\alpha}^{-1/2}
      \left( -\log \bar{G}(x) \right)^{\frac{1}{2}\left( \alpha-1 \right)}
     \left(\bar{G}(x)\right)^{\alpha}.
   $$
Hence, $\bar{F}(x)$ and $\bar{G}(x)$ satisfy the assumption of
Theorem \ref{theorem:asympN2} with
$$
\Phi(x) =   {\alpha}^{1/2}
      (\pi \log x)^{\frac{1}{2}\left( 1-\alpha
     \right)}x^{\alpha},
$$
which implies (\ref{eq:normal}). \qed
\end{proof}

Next, we present the proofs for the preceding Propositions
\ref{eq:proposition1}, \ref{eq:proposition2} and Theorem
\ref{theorem:asympN2}. Note that the following proof represents a
basis for the other proofs in this paper.

\begin{proof}[of Proposition \ref{eq:proposition1}]
 Notice that the number of retransmissions is geometrically
 distributed given the packet size $L$,
 $$ \Pr[N>n \mid L]=(1-\bar{G}(L))^n
 $$
and, therefore,
 \begin{equation}\label{eq:rep}
  \Pr[N>n]=\expect[(1-\bar{G}(L))^n].
\end{equation}

Since $\Phi(x)$ is eventually non-decreasing, there exists $x_{0}$
such that for all $x>x_0$, $\Phi(x)$ has an inverse function
$\Phi^{\leftarrow}(x)$.  The condition (\ref{eq:propo1A}) implies
that, for $0<\epsilon<1$,  there exists $x_{\epsilon}$, such that
for $x>x_{\epsilon}$,
$$(1-\epsilon) \bar{F}^{-1}(x) \leq
\Phi( \bar{G}^{-1}(x))\leq (1+\epsilon) \bar{F}^{-1}(x),$$ and thus,
by choosing $x_{\epsilon} > x_0$, we obtain
\begin{equation}\label{eq:equivBound2}
\Phi^{\leftarrow}\left( (1-\epsilon) \bar{F}^{-1}(x) \right) \leq
\bar{G}^{-1}(x)\leq  \Phi^{\leftarrow} \left((1+\epsilon)
\bar{F}^{-1}(x) \right).
\end{equation}
First, we will prove the \emph{upper bound}. Recalling
(\ref{eq:rep}), noting that $V\eqdef \bar{F}(L)$ is a uniform random
variable
 on $(0,1)$ (e.g., see Proposition 2.1 in
Chapter 10 of \cite{Ross02}) and using (\ref{eq:equivBound2}),  we
obtain, for large $n$,
\begin{align}\label{eq:propupper1}
  \Pr[N>n]& = \expect[(1-\bar{G}(L))^n] \nonumber\\
  &= \expect[(1-\bar{G}(L))^n
  \ind(L> x_{\epsilon})]  +\expect[\left(1-\bar{G}(L)\right)^n
  \ind(L \leq x_{\epsilon})] \nonumber\\
  &\leq \expect \left[  e^{- \frac{n}{\Phi^{\leftarrow}\left((1+\epsilon)V^{-1}\right)} } \right] +
  \left(1-\bar{G}(x_{\epsilon}) \right)^n  \nonumber\\
  & \leq \Pr\left[0\leq  \frac{n}{\Phi^{\leftarrow}\left((1+\epsilon)V^{-1}\right)} \leq 1 \right]
      + \sum_{k=0}^{\lceil \log (\epsilon n) \rceil} e^{-e^{k}}
        \Pr\left[ e^k \leq  \frac{n}{\Phi^{\leftarrow}\left((1+\epsilon)V^{-1}\right)} \leq e^{k+1} \right]
      \nonumber\\
      &\;\;\; + e^{-e^{\lceil \log (\epsilon n) \rceil+1}}
        \Pr\left[ \frac{n}{\Phi^{\leftarrow}\left((1+\epsilon)V^{-1}\right)} > e^{\lceil \log (\epsilon n) \rceil+1} \right]  +
  \left(1-\bar{G}(x_{\epsilon}) \right)^n  \nonumber\\
  &\leq  \frac{1+\epsilon}{\Phi\left(n\right)}
   + \sum_{k=0}^{\lceil \log (\epsilon n) \rceil} e^{-e^{k}}
          \frac{1+\epsilon}{\Phi \left(  \frac{n}{e^{k+1}} \right)}
      + e^{-\epsilon n}  +
  \left(1-\bar{G}(x_{\epsilon}) \right)^n.
\end{align}
The condition (\ref{eq:dominantRV}) implies that there exist finite
$n_d$ and $d$, such that for  $n>n_d$,
  $$ \frac{\Phi(n)}{\Phi(n/e)} < d,$$
resulting in, for all $k$ satisfying $n/e^{k} > n_d$,
\begin{equation}\label{eq:prop1c}
     \frac{ \Phi \left( n \right) }{\Phi \left(  \frac{n}{e^{k+1}} \right)} \leq
  d^{k+2},
  \end{equation}
  and therefore,
  $$ \Phi(n) \leq \Phi(n_d) d^{\log \left(\frac{n}{n_d} \right)+1},$$
  which, in conjunction with (\ref{eq:propupper1}), yields
\begin{align}\label{eq:prop1b}
 \varlimsup_{n \to \infty} \Pr[N>n] \Phi(n) &\leq  1+\epsilon
   + \sum_{k=0}^{\infty}(1+\epsilon) e^{-e^{k}}
          d^{k+2} \nonumber \\
 & \;\;\;\; + \varlimsup_{n \to \infty} \left ( e^{-\epsilon n}  +
  \left(1-\bar{G}(x_{\epsilon}) \right)^n \right) \Phi(n_d) d^{\log
  \left(\frac{n}{n_d} \right)+1} \nonumber\\
  & = 1+\epsilon
   + \sum_{k=0}^{\infty}(1+\epsilon) e^{-e^{k}}
          d^{k+2}<\infty.
\end{align}

Next, we prove the \emph{lower bound}. Recalling
(\ref{eq:equivBound2}) and choose $n>x_0$, we obtain
\begin{align}
 \Pr[N>n]& = \expect[(1-\bar{G}(L))^n] \nonumber\\
 &\geq  \left( 1- \frac{1}{n} \right)^n \Pr \left[\bar{G}(L) \leq \frac{1}{n}
 \right] \nonumber\\
 &\geq  \left( 1- \frac{1}{n} \right)^n \Pr \left[\Phi^{\leftarrow} \left((1-\epsilon) \bar{F}^{-1}(L)
\right) \geq  n
 \right] \nonumber\\
 &\geq \left( 1- \frac{1}{n} \right)^n \frac{1-\epsilon}{\Phi(n)} \; ,
 \nonumber
\end{align}
implying
\begin{align}
 \varliminf_{n \to \infty} \Pr[N>n] \Phi(n)  &\geq  \lim_{n \to \infty}\left( 1- \frac{1}{n} \right)^n
 (1-\epsilon)=e^{-1}(1-\epsilon),\nonumber
\end{align}
which, in conjunction with (\ref{eq:prop1b}), proves the
proposition.\qed
\end{proof}

\begin{proof}[of Proposition \ref{eq:proposition2}]
 Recalling (\ref{eq:equivBound2}) and
choosing $n$ large enough such that $\{\bar{G}(L)\leq e/n \}
\subseteq \{L>x_{\epsilon}\}$ with $x_{\epsilon}$ being the same as
chosen in (\ref{eq:equivBound2}), we obtain
\begin{align}
 \Pr[N>n]& = \expect[(1-\bar{G}(L))^n] \nonumber\\
 &\geq  \left( 1- \frac{e}{n} \right)^n \Pr \left[\bar{G}(L) \leq \frac{e}{n}
 \right] \nonumber\\
 &\geq  \left( 1- \frac{e}{n} \right)^n \Pr \left[\Phi^{\leftarrow} \left((1-\epsilon) \bar{F}^{-1}(L)
\right) \geq  \frac{n}{e}
 \right] \nonumber\\
 &\geq \left( 1- \frac{e}{n} \right)^n \frac{1-\epsilon}{\Phi \left(\frac{n}{e}\right)} \; ,
 \nonumber
\end{align}
implying
\begin{align}
 \varliminf_{n \to \infty} \Pr[N>n] \Phi(n)  &\geq  \varliminf_{n \to \infty}\left( 1- \frac{1}{n} \right)^n
 \frac{(1-\epsilon)\Phi(n)}{\Phi \left(\frac{n}{e}\right)}= \infty,\nonumber
\end{align}
which completes the proof. \qed
\end{proof}


%
%

\begin{proof}[of Theorem \ref{theorem:asympN2}]
 We begin with proving (\ref{eq:asympNa}). Without loss of generality,  we can assume
that $\Phi(x)$ is absolutely continuous and strictly monotone since,
by Proposition 1.5.8 of \cite{BG87}, one can always find an
absolutely continuous and strictly monotone function
\begin{equation}\label{eq:modified}
\Phi^{\ast}(x)= \alpha \int_{1}^{x}\Phi(s)s^{-1}ds, \; x\geq 1,
\end{equation}
which satisfies
$$\bar{F}^{-1}(x) \thicksim \Phi\left( \bar{G}^{-1}(x)\right)
  \thicksim \Phi^{\ast}\left( \bar{G}^{-1}(x)\right).$$

First, we prove the \emph{upper bound}. Recalling (\ref{eq:rep}),
noting that $V\eqdef \bar{F}(L)$ is a uniform random variable on
$(0,1)$ and using (\ref{eq:equivBound2}),  we obtain
\begin{align}
  \Pr[N>n]& = \expect[(1-\bar{G}(L))^n] \nonumber\\
  &= \expect[(1-\bar{G}(L))^n
  \ind(L \geq x_{\epsilon})]  +\expect[\left(1-\bar{G}(L)\right)^n
  \ind(L < x_{\epsilon})] \nonumber\\
  &\leq \expect \left[  e^{- \frac{n}{\Phi^{(-1)}\left((1+\epsilon) V^{-1}\right)} }
               \right] +
  \left(1-\bar{G}(x_{\epsilon}) \right)^n.
  \end{align}
 Then, by choosing integers $m$ and $n_d$ (as in (\ref{eq:prop1c})) and
  noting that $\Phi(n)$ is regularly varying,  the preceding inequality yields, for
  large $n$,
  \begin{align}
  \Pr[N>n] & \leq \expect\left[ e^{- \frac{n}{\Phi^{(-1)}\left((1+\epsilon) V^{-1}\right)} }
    \ind \left( 0\leq  \frac{n}{\Phi^{\leftarrow}\left((1+\epsilon)V^{-1}\right)} \leq e^m \right) \right]
      \nonumber\\
      &\;\;\;\;\; + \sum_{k=m}^{\log \left( \frac{n}{n_d} \right)-1} e^{-e^{k}}
        \Pr\left[ e^k \leq  \frac{n}{\Phi^{\leftarrow}\left((1+\epsilon)V^{-1}\right)} \leq e^{k+1} \right]
             + o\left( \frac{1}{\Phi(n)} \right)  \nonumber\\
  &\leq \int_{0}^{e^m} e^{-z} \left(\frac{\Phi'\left( n/z \right)}
                          {\Phi^2\left(n/z \right)}
        \frac{(1+\epsilon)n}{z^2} \right) dz
   + \sum_{k=m}^{\log \left( \frac{n}{n_d} \right)-1} e^{-e^{k}}
          \frac{1+\epsilon}{\Phi \left(  \frac{n}{e^{k+1}} \right)}
      + o\left( \frac{1}{\Phi(n)} \right), \nonumber
 \end{align}
resulting in
\begin{align}\label{eq:propupper5}
\Pr[N>n]\Phi(n)& \leq  \int_{0}^{e^m} \frac{\Phi(n)}
                          {\Phi\left(n/z \right)} \frac{\Phi'\left( n/z \right)}
                          {\Phi\left(n/z \right)}
        \frac{ e^{-z}(1+\epsilon)n}{z^2}dz \nonumber \\
        &\;\;\;\;\; +
          \sum_{k=m}^{\log \left( \frac{n}{n_d} \right)-1} (1+\epsilon)e^{-e^{k}}
          \frac{\Phi(n)}{\Phi \left(  \frac{n}{e^{k+1}} \right)}
          + \Phi(n) \left(1-\bar{G}(x_{\epsilon})
        \right)^n\nonumber\\
        &\eqdef I_1+I_2+I_3.
\end{align}
Since regularly varying functions are also dominantly regularly
varying, the bound in (\ref{eq:prop1c}) implies
\begin{equation}\label{eq:propupper6} I_2 \leq  \sum_{k=m}^{\log
\left( \frac{n}{n_d} \right)-1} (1+\epsilon)e^{-e^{k}} d^{k+2} \leq
\sum_{k=m}^{\infty} (1+\epsilon)e^{-e^{k}} d^{k+2}<\infty.
\end{equation}
For $I_1$, since $\Phi(n)$ is regularly varying, by the
Characterisation Theorem of regular variation (e.g., see Theorem
1.4.1 of
 \cite{BG87}) and the uniform convergence theorem of slowly varying
 functions (Theorem 1.2.1 of \cite{BG87}),  it is easy to
obtain uniformly for $0\leq z\leq e^m$, as $n\to \infty$,
$$
  \frac{\Phi(n)}
                          {\Phi\left(n/z \right)}
                          \sim
                         z^{\alpha}
$$
and, recalling (\ref{eq:modified}),
$$
\frac{\Phi'\left( n/z \right)}
                          {\Phi\left(n/z \right)}
                          = \frac{z\alpha}{n},
$$
which implies
\begin{equation}\label{eq:propupper5I1}
I_1 \sim \int_{0}^{e^m}(1+\epsilon)\alpha e^{-z} z^{\alpha-1}dz.
\end{equation}
Furthermore,  $\Phi(n)$ being regularly varying implies that $I_3
\to 0$ as $n \to \infty$. Thus, passing $n \to \infty $ in
(\ref{eq:propupper5}), recalling (\ref{eq:propupper6}) and then
passing $m \to \infty$, $\epsilon \to 0$, we obtain
  \begin{equation}\label{eq:asympN5}
   \Pr[N>n]\Phi(n) \lesssim \int_{0}^{\infty}\alpha e^{-z}
z^{\alpha-1}dz=\Gamma(\alpha+1).
\end{equation}

As for the \emph{lower bound}, the proof follows similar arguments,
and the details  are presented in Subsection \ref{ss:continuation}.
The same subsection also contains the proof of the statement ii) of
the theorem. \qed
\end{proof}

The condition of $\Phi(\cdot)$ being dominantly varying is basically
necessary in order for $\Pr[N>n] \approx \Phi(n)^{-1}$ to hold.  As
shown in Proposition \ref{prop:gtrOne}, this tail equivalence
basically does not hold if $\Phi(\cdot)$ is not dominantly varying,
e.g., if $\Phi(\cdot)^{-1}$ is lognormal. Here, we further
characterize the behavior of the lognormal type distributions in the
following proposition.
\begin{proposition}\label{prop:gtrOne}
If $\log(\Phi(x))= \lambda (\log x)^{\delta}$, $\delta>1,
\lambda>0$, then, under the condition (\ref{eq:propo1A}), we obtain
$$
  \lim_{n \to \infty} \frac{\log \left(\Pr[N>n]^{-1}\right) - \log(\Phi(x))}
     {(\log \log n)(\log n)^{\delta-1}} = -\lambda \delta(\delta-1).
$$
\end{proposition}
The \textbf{proof} of this proposition is presented in Subsection
\ref{ss:gtrOne}.
\begin{remark}
In Proposition \ref{prop:gtrOne}, it can be easily verified that
$\Phi(\cdot)$ is not dominantly regularly varying, and therefore,
according to Propositions \ref{eq:proposition1} and
\ref{eq:proposition2}, we know $\Pr[N>n]\Phi(x) \to \infty$ as $n\to
\infty$. However, Proposition \ref{prop:gtrOne} further characterize
how fast $\Pr[N>n]\Phi(n)$ goes to infinity in the logarithmic
scale, which also implies a weaker result
$$
   \lim_{n \rightarrow \infty}\frac{ \log \left(\Pr[N>n]^{-1}\right)}{ \log(\Phi(n))}=1.
$$
In the following theorem we extend the preceding logarithmic limit
under a more general condition on $\Phi(\cdot)$.
\end{remark}

\begin{theorem}\label{theorem:1I}
  If an eventually non-decreasing function $\Phi(x)\eqdef e^{l(x)}$ satisfies (\ref{eq:conditionI2})
  where $l(x)$ is slowly varying with
 \begin{equation}\label{eq:lx}
  \lim_{x \to
  \infty} \frac{l\left( \frac{x}{l(x)} \right)}{l(x)} = 1,
  \end{equation}
    then,
   \begin{equation}\label{eq:NI}
   \lim_{n \rightarrow \infty}\frac{ \log \left(\Pr[N>n]^{-1}\right)}{ \log\Phi(
   n)}=1.
   \end{equation}
\end{theorem}

\begin{remark}\label{remark:1}
Note that if $\log (\Phi(x)) = e^{(\log x)^{\delta}}$ then the
condition (\ref{eq:lx}) is satisfied if $0<\delta<1/2$ and it does
not hold if $\delta \geq 1/2$, which can be easily verified.
Furthermore, if $\log (\Phi(x)) = \Psi(\log x)$ where $\Psi(x)$ is
regularly varying, e.g., $\Phi(x)^{-1}$ being lognormal, then the
condition (\ref{eq:lx}) also holds, which is stated in the following
corollary.
\end{remark}

\begin{remark}[Lattice variables]\label{remark:lattice2}
When $L$ is lattice valued,  it is easy to see from the proof of
Theorem \ref{theorem:1I} that, if there exists a continuous random
variable $L^{\ast}$ such that $\log \Pr[L^{\ast}>x] \sim \log
\Pr[L>x]$ as $x \to \infty$, or equivalently, if there exists a
continuous negative non-increasing function $q(x)$ such that $\log
\Pr[L>x] \sim q(x)$, then Theorem \ref{theorem:1I} still holds,
e.g., when $L$ has a geometric or Poisson distribution. To
rigorously prove this claim, one can use similar arguments as in the
proof of Theorem \ref{theorem:tamdem} in Section \ref{s:APP} of this
paper. Note that this remark also applies to other logarithmic
asymptotics, e.g., see Corollary \ref{corollary:1}, Propositions
\ref{prop:betweenHalfOne} and \ref{theorem:nearlyB}, and Theorems
\ref{theorem:2I}, \ref{theorem:nearly}, \ref{theorem:T} and
\ref{theorem:TWeibull}.
\end{remark}

\begin{corollary}\label{corollary:1}
  If  a regularly varying function $\Psi(\cdot)$ with a non-negative index  
satisfies
\begin{equation}
   \lim_{x\rightarrow \infty}\frac{\log \bar{F}(x)^{-1}}{\Psi \left(\log
   \bar{G}(x)^{-1} \right)}=1 \nonumber
   \end{equation}
 and, in addition,  is eventually non-decreasing when $\Psi(\cdot)$ is slowly varying,   then, we have
   \begin{equation}
   \lim_{n \rightarrow \infty}\frac{ \log \left(\Pr[N>n]^{-1}\right)}{ \Psi(\log
   n)}=1. \nonumber
   \end{equation}
\end{corollary}
\begin{remark}
This result, or more precisely Theorem \ref{theorem:T} in Subsection
\ref{ss:asmpT}, implies parts (1:1), (2:1) and (2:2) of Theorem 2.1
in \cite{Asmussen07} and extends Theorem 2 in \cite{PJ07RETRANS}.
\end{remark}
\begin{proof}[of Corollary \ref{corollary:1}]
 For a regularly varying function $\Psi(\cdot)$,  it is easy to verify that
    $l(x) = \Psi (\log (x))$ satisfies
     \begin{equation}
  \lim_{x \to
  \infty} \frac{l\left( \frac{x}{l(x)} \right)}{l(x)} =
   \lim_{x \to \infty}\frac{\Psi \left(\log x - \log \Psi (\log(x)) \right)}{ \Psi (\log(x))}=1, \nonumber
  \end{equation}
  and therefore, by Theorem \ref{theorem:1I}, we prove the
  corollary.
   \qed
    \end{proof}
\begin{remark}
Note that, in conjunction with Remark \ref{remark:1}, the condition
(\ref{eq:lx}) is close to necessary since the result (\ref{eq:NI})
does not hold if $\log ( \Phi(x)) = e^{(\log x)^{\delta}} $,
$1/2<\delta<1$, as can be seen from the following proposition.
\end{remark}

\begin{proposition}\label{prop:betweenHalfOne}
If $\log (\Phi(x)) =e^{\lambda (\log x)^{\delta}}$, $1/2<\delta<1$,
$\lambda>0$, then, under the condition (\ref{eq:conditionI2}), we
obtain
$$
  \log \left(\log \left(\Pr[N>n]^{-1}\right)\right) - \log \left(\log (\Phi(x)) \right) \sim - \delta \lambda^2 (\log
  n)^{2\delta-1}.
$$
\end{proposition}
The \textbf{proof} of this proposition is presented in Subsection
\ref{ss:betweenHalfOne}.
\begin{remark}
Note that this result implies that, for $0<\epsilon<1$ and $n$
large,
  \begin{equation}
 0\leq \frac{ \log \left(\Pr[N>n]^{-1}\right)}{ \log\Phi(
   n)} \leq  e^{- (1-\epsilon)\alpha \lambda^2 (\log
  n)^{2\alpha-1} } \to 0, \nonumber
   \end{equation}
   which contrasts the limit in (\ref{eq:NI}).
\end{remark}

\begin{proof}[of Theorem \ref{theorem:1I}]
Since $\Phi(x)$ is eventually non-decreasing, there exists $x_{0}$
 such that for all $x>x_0$, $\Phi(x)$ has an inverse function
$\Phi^{\leftarrow}(x)$.  The condition (\ref{eq:conditionI2})
implies that, for $0<\epsilon<1$,  there exists $x_{\epsilon}$, such
that for $x>x_{\epsilon}$,
$$ \bar{F}^{-(1-\epsilon)}(x) \leq
\Phi( \bar{G}^{-1}(x))\leq \bar{F}^{-(1+\epsilon)}(x),$$
 thus, choosing $x_{\epsilon} > x_0$, we obtain
\begin{equation}\label{eq:equivBound2I}
\Phi^{\leftarrow}\left(  \bar{F}^{-(1-\epsilon)}(x) \right) \leq
\bar{G}^{-1}(x)\leq \Phi^{\leftarrow} \left(
\bar{F}^{-(1+\epsilon)}(x) \right).
\end{equation}
First, we prove the \emph{upper bound}. Recalling (\ref{eq:rep}),
noting that $V\eqdef \bar{F}(L)$ is a uniform random variable on
$(0,1)$, and using (\ref{eq:equivBound2I}),  we obtain, for integer
$y$ and large $n$,
\begin{align}
  \Pr[N>n]& = \expect[(1-\bar{G}(L))^n] \nonumber\\
  &= \expect[(1-\bar{G}(L))^n
  \ind(L> x_{\epsilon})]  +\expect[\left(1-\bar{G}(L)\right)^n
  \ind(L \leq x_{\epsilon})] \nonumber\\
  &\leq \expect \left[  e^{- \frac{n}{\Phi^{\leftarrow}\left(V^{-(1+\epsilon)}\right)} } \right] +
  \left(1-\bar{G}(x_{\epsilon}) \right)^n  \nonumber\\
  & \leq \sum_{k=0}^{y} e^{-k}
        \Pr\left[ k \leq  \frac{n}{\Phi^{\leftarrow}\left(V^{-(1+\epsilon)}\right)} \leq k+1 \right]
        + e^{-(y+1)}+
  \left(1-\bar{G}(x_{\epsilon}) \right)^n, \nonumber
\end{align}
which, by Proposition \ref{prop:sub}, noting $\Phi(x)= e^{l(x)}$ and
choosing $y=\lceil l(n) \rceil-1$, implies
\begin{align}\label{eq:theom1Iupper}
  \Pr[N>n]&  \leq   \sum_{k=0}^{\lceil l(n) \rceil-1} e^{-k- \frac{1}{1+\epsilon}l\left(  \frac{n}{k+1} \right)} + e^{-l(n)}+
  o\left( \Pr[N>n] \right) \nonumber\\
  &\leq \lceil l(n) \rceil e^{-\frac{1}{1+\epsilon}l\left( \frac{n}{\lceil l(n) \rceil} \right)}
  + e^{-l(n)}+
  o\left( \Pr[N>n] \right).
\end{align}
From (\ref{eq:conditionI2}), it is easy to see  that $l(x)$
increases to infinity when $x\to \infty$ and,  since $l(x)$ is
slowly varying, by (\ref{eq:lx}) and (\ref{eq:theom1Iupper}), we
obtain
\begin{align}\label{eq:theom1Iupper2}
 \varliminf_{n \to \infty} \frac{ \log \Pr[N>n]^{-1}}{ l(n)} &  \geq  1.
\end{align}

Next, we prove the \emph{lower bound}. Recalling
(\ref{eq:equivBound2I}) and choosing $n$ large enough, we obtain
\begin{align}
 \Pr[N>n]& = \expect[(1-\bar{G}(L))^n] \nonumber\\
 &\geq  \left( 1- \frac{1}{n} \right)^n \Pr \left[\bar{G}(L) \leq \frac{1}{n}
 \right] \nonumber\\
 &\geq  \left( 1- \frac{1}{n} \right)^n \Pr \left[\Phi^{\leftarrow} \left( \bar{F}^{-(1-\epsilon)}(L)
\right) \geq  n
 \right] \nonumber\\
 &\geq \left( 1- \frac{1}{n} \right)^n \frac{1}{\Phi(n)^{\frac{1}{1-\epsilon}}} \; ,
 \nonumber
\end{align}
implying
\begin{align}
 \varlimsup_{n \to \infty} \frac{ \log \Pr[N>n]^{-1}}{ l(n)} &  \leq
 \frac{1}{1-\epsilon}, \nonumber
\end{align}
which, by passing $\epsilon\to 0$ and in conjunction with
(\ref{eq:theom1Iupper2}), proves the theorem. \qed
\end{proof}

\subsubsection{Medium Heavy (Weibull) Asymptotics }\label{ss:weibull}
 In the
preceding subsection, we studied the scenario when the distribution
of $N$ is heavier than any Weibull distribution. Specifically, we
establish the necessary conditions under which $\Pr[N>n] \approx
\Phi^{-1}(n)$ holds when the separation between $\Pr[L>x]$ and
$\Pr[A>x]$ can be characterized in the form of $\Phi(x)=e^{l(x)}$
with $l(x)$ being slowly varying. In this subsection, we further
increase the separation in the sense that $\Phi(x)=e^{R_{\beta}(x)}$
with $R_{\beta}(x)$ being regularly varying of index $\beta>0$, and
under this condition the distribution of $N$ is shown to be of
Weibull type.  In this situation, the tail equivalence developed in
the preceding subsection does not hold anymore and admits a
different form, as stated in the following theorem.

\begin{theorem}\label{theorem:2I}
  If an eventually non-decreasing function $\Phi(x)\eqdef e^{R_{\beta}(x)}$ satisfies (\ref{eq:conditionI2})
    where $R_{\beta}(x)\equiv x^{\beta}l(x)$, $\beta>0$
   is regularly varying with $l(x)$ satisfying
 \begin{equation}\label{eq:lx2}
  \lim_{x \to
  \infty} \frac{l\left( \left( \frac{x}{l(x)} \right)^{\frac{1}{1+\beta}} \right)}{l(x)} = 1,
  \end{equation} 
    then,
   \begin{equation}\label{eq:NI2}
   \lim_{n \rightarrow \infty}\frac{ \log \left(\Pr[N>n]^{-1} \right)}{ \left(\log\Phi(
   n) \right)^{\frac{1}{\beta+1}}}=\beta^{\frac{1}{\beta+1}} +
\beta^{-\frac{\beta}{\beta+1}}.
   \end{equation}
\end{theorem}
\begin{remark}
This theorem, or more precisely Theorem \ref{theorem:TWeibull} of
the following Subsection \ref{ss:asmpT},  implies part (1:2) of
Theorem 2.1 in \cite{Asmussen07}, and provides a more precise
logarithmic asymptotics instead of a double logarithmic limit that
was proved in \cite{Asmussen07}. Furthermore, although the condition
(\ref{eq:lx2}) appears complicated, it is easy to check that any
slowly varying function $l(x)=l_1(\log x)$ satisfies it, where
$l_1(\cdot)$ is also a slowly varying function.
\end{remark}

\begin{proof}[of Theorem \ref{theorem:2I}]
First, we begin with proving the \emph{upper bound}. Following the
same approach as in the proof of Theorem \ref{theorem:1I}, we
obtain, for $\epsilon>0$, integer $y$ and $n$ large enough,
\begin{align}\label{eq:theom2Iupper}
  \Pr[N>n]  & \leq \sum_{k=0}^{y-1} e^{-k}
        \Pr\left[ k \leq  \frac{n}{\Phi^{\leftarrow}(V^{-(1+\epsilon)})} \leq k+1 \right]
        + e^{-y}+
 o\left( \Pr[N>n] \right) \nonumber\\
 &\leq  \sum_{k=0}^{ y-1 } e^{-k- \frac{1}{1+\epsilon} R_{\beta}\left(  \frac{n}{k+1} \right)}
 + e^{-y}+
  o\left( \Pr[N>n] \right).
\end{align}
Using the same argument as in (\ref{eq:modified}), we can find an
absolutely continuous and
 strictly increasing function
$R^{\ast}_{\beta}(u)\eqdef \beta \int_{1}^{u}R_{\beta}(s)s^{-1}ds,
u\geq 1$ that is a modified version of $R_{\beta}(u)$. This newly
constructed function $R^{\ast}_{\beta}(u)$ satisfies that, for
$0<\epsilon<1$, there exists $y_{\epsilon}>0$, such that $
(1-\epsilon)R^{\ast}_{\beta}(u)<R_{\beta}(u)<(1+\epsilon)R^{\ast}_{\beta}(u)$
for $u>y_{\epsilon}$. Therefore, for $0<x<n/y_{\epsilon}$,
$$x+ \frac{1}{1+\epsilon} R_{\beta}\left(\frac{n}{x}\right)
 \geq x+ \frac{1-\epsilon}{1+\epsilon}
 R^{\ast}_{\beta}\left(\frac{n}{x}\right),
 $$
and, for $u\geq 1$,
\begin{equation}\label{eq:theom2IupperC}
\left(R^{\ast}_{\beta}(u)\right)' = \beta u^{\beta-1} l(u).
\end{equation}

 Choosing $y= \lceil n/y_{\epsilon} \rceil$ in
(\ref{eq:theom2Iupper}) and using the asymptotic equivalence
relationship between $R_{\beta}(\cdot)$ and
$R^{\ast}_{\beta}(\cdot)$, we obtain
\begin{align}
  \Pr[N>n]  & \leq \sum_{k=0}^{ \left\lceil \frac{n}{y_{\epsilon}} \right\rceil-1 }
      e^{-k- \frac{1}{(1+\epsilon)} R_{\beta}\left(  \frac{n}{k+1} \right)}
 + e^{-\frac{n}{y_{\epsilon}}}+
  o\left( \Pr[N>n] \right) \nonumber\\
  & \leq \sum_{k=0}^{ \left\lceil \frac{n}{y_{\epsilon}} \right\rceil-1 }
      e^{-k- \frac{1-\epsilon}{1+\epsilon} R^{\ast}_{\beta}\left(  \frac{n}{k+1} \right)}
      +  o\left( \Pr[N>n] \right).\label{eq:theom2IupperB}
\end{align}
Next, let $f(x)= x+  R^{\ast}_{\beta}\left( n/x
\right)(1-\epsilon)/(1+\epsilon)$, and suppose that $f(x)$ reaches
the maximum at $x^{\ast}$ for $0< x \leq n/y_{\epsilon}$. From
(\ref{eq:theom2IupperC}), it is easy to check that
$$
f'(x)= 1- \frac{1-\epsilon}{1+\epsilon}
          \left( R^{\ast}_{\beta}\left(\frac{n}{x}\right) \right)' \frac{n}{x^2} = 1-
          \frac{1-\epsilon}{1+\epsilon} \frac{\beta}{n} \frac{
          n^{\beta+1}}{x^{\beta+1}}l\left(\frac{n}{x} \right).
$$
Then, define $g(u)\eqdef
          u^{\beta+1}l\left(u\right)$, and use the same argument as
          in constructing $R^{\ast}_{\beta}(\cdot)$,
       we can find an
absolutely continuous and
 strictly increasing function
$g^{\ast}(u)\eqdef \beta \int_{1}^{u}u^{\beta}l\left(u\right)ds,
u\geq 1$, such that $(1-\epsilon) g(u)
<g^{\ast}(u)<(1+\epsilon)g(u), u>u_{\epsilon} $ for
$u_{\epsilon}>0$. Therefore, for $0<x<n/u_{\epsilon}$, we obtain,
$$
  1- \frac{1}{1+\epsilon} \frac{\beta}{n}g^{\ast}\left(\frac{n}{x}\right)
  < f'(x) = 1- \frac{1-\epsilon}{1+\epsilon} \frac{\beta}{n}g\left(\frac{n}{x}\right)
      < 1- \frac{1-\epsilon}{(1+\epsilon)^2}
  \frac{\beta}{n}g^{\ast}\left(\frac{n}{x}\right),
$$
where, as shown in the preceding inequalities,  the lower and upper
bound of $f'(x)$ are two monotonically increasing functions for
$0<x<n$.

Now, define
$$x_1\eqdef \left(\frac{(1-\epsilon)^3}{(1+\epsilon)^2} \right)^{\frac{1}{\beta+1}}
\beta^{\frac{1}{\beta+1}} n ^{\frac{\beta}{\beta+1}} l\left( n
\right)^{\frac{1}{\beta+1}}
$$
and
$$
 x_2\eqdef
(1+\epsilon)^{\frac{1}{\beta+1}}\beta^{\frac{1}{\beta+1}}
n^{\frac{\beta}{\beta+1}} l\left( n \right)^{\frac{1}{\beta+1}}.
$$
It is easy to see that,  by condition (\ref{eq:lx2}), for $n$ large
enough,
$$
 f'(x_1)\leq 1- \frac{1-\epsilon}{(1+\epsilon)^2}
  \frac{\beta}{n}g^{\ast}\left(\frac{n}{x_1}\right)< 1-
  \left(\frac{1-\epsilon}{1+\epsilon}\right)^2
  \frac{\beta}{n}g\left(\frac{n}{x_1}\right)<0,
$$
and
$$
 f'(x_2)\geq 1- \frac{1}{1+\epsilon}
  \frac{\beta}{n}g^{\ast}\left(\frac{n}{x_2}\right)> 1-
  \frac{\beta}{n}g\left(\frac{n}{x_2}\right)>0,
$$
 which implies that, there exist $n_{\epsilon}>0$ such that
 for all $n>n_{\epsilon}$,
\begin{equation}\label{eq:x0}
x_1< x^{\ast} < x_2.
\end{equation}
Therefore, using (\ref{eq:theom2IupperB}), (\ref{eq:x0}) and
recalling $R_{\beta}(u)<(1+\epsilon)R^{\ast}_{\beta}(u)$ yields
\begin{align}
  \Pr[N>n]  &\leq  \left\lceil \frac{n}{y_{\epsilon}}
\right\rceil e^{1-f(x^{\ast})} +
  o\left( \Pr[N>n] \right) \nonumber\\
  & \leq \left\lceil \frac{n}{y_{\epsilon}}
\right\rceil  e^{1-x_1-\frac{1}{(1+\epsilon)^2} R_{\beta}\left(
\frac{n}{x_2} \right)} +
  o\left( \Pr[N>n] \right), \nonumber
\end{align}
resulting in
\begin{equation}
\varliminf_{n \to \infty} \frac{\log
\left(\Pr[N>n]^{-1}\right)}{n^{\frac{\beta}{\beta+1}}
l(n)^{\frac{1}{\beta+1}} } \geq
\left(\frac{(1-\epsilon)^3}{(1+\epsilon)^2}
\right)^{\frac{1}{\beta+1}}\beta^{\frac{1}{\beta+1}} +
\left(1+\epsilon \right)^{-\frac{\beta}{\beta+1}-2}
\beta^{-\frac{\beta}{\beta+1}}. \nonumber
\end{equation}
Passing $\epsilon\to 0$ in the preceding inequality yields
\begin{equation}\label{eq:theom2Iupper3}
\varliminf_{n \to \infty} \frac{\log
\left(\Pr[N>n]^{-1}\right)}{n^{\frac{\beta}{\beta+1}}
l(n)^{\frac{1}{\beta+1}} } \geq \beta^{\frac{1}{\beta+1}} +
\beta^{-\frac{\beta}{\beta+1}}.
\end{equation}

Now, we proceed with proving the \emph{lower bound}. By recalling
the condition (\ref{eq:equivBound2I}) and using $1-x\geq
e^{-(1+\epsilon)x}$ for $x$ small enough, we obtain, for $n$ large
enough and $x_0>0$,
\begin{align}
  \Pr[N>n] &\geq \expect[(1-\bar{G}(L))^n
  \ind(L\geq x_{\epsilon})] \geq \expect[e^{-(1+\epsilon)\bar{G}(L)n}
  \ind(L\geq x_{\epsilon})] \nonumber\\
  &\geq \expect \left[  e^{- \frac{(1+\epsilon)n}{\Phi^{\leftarrow}\left(V^{-(1-\epsilon)}\right)} }
   \ind(V\leq \bar{F}(x_{\epsilon})) \right]  \geq  e^{-x_0}
        \Pr\left[  \frac{(1+\epsilon)n}{\Phi^{\leftarrow}\left(V^{-(1-\epsilon)}\right)} \leq x_0,
        V\leq \bar{F}(x_{\epsilon}) \right]
        \nonumber\\
  & = e^{-x_0 }\Phi\left(\frac{(1+\epsilon)n}{x_0} \right)^{-\frac{1}{1-\epsilon}}
   = e^{-x_0-\frac{1}{1-\epsilon}R_{\beta}\left(\frac{(1+\epsilon)n}{x_0}
   \right)}, \nonumber
\end{align}
since
$\{(1+\epsilon)n/\Phi^{\leftarrow}\left(V^{-(1-\epsilon)}\right)
\leq x_0 \}$ implies $\{V\leq \bar{F}(x_{\epsilon})\}$  for all $n$
large enough. Next, by choosing $x_0=\beta^{\frac{1}{\beta+1}} n
^{\frac{\beta}{\beta+1}} l\left( n \right)^{\frac{1}{\beta+1}}$,
 using the condition (\ref{eq:lx2}), and then passing $n\to \infty$
 as well as $\epsilon \to 0$, yields,
\begin{equation}\label{eq:theom2Ilower3}
\varlimsup_{n \to \infty} \frac{\log
\left(\Pr[N>n]^{-1}\right)}{n^{\frac{\beta}{\beta+1}}
l(n)^{\frac{1}{\beta+1}} } \leq \beta^{\frac{1}{\beta+1}} +
\beta^{-\frac{\beta}{\beta+1}}.
\end{equation}
Finally, combining (\ref{eq:theom2Iupper3}) and
(\ref{eq:theom2Ilower3}) finishes the proof. \qed
\end{proof}

\subsubsection{Nearly Exponential Asymptotics}\label{ss:nearly}
 In the
preceding subsection,  the functional separation between $\Pr[L>x]$
and $\Pr[A>x]$ can be characterized in the form of
$\Phi(x)=e^{R_{\gamma}(x)}$ with $R_{\gamma}(x)$ being regularly
varying. In this subsection, we investigate the situation when the
separation in terms of $\Phi(x)$ is even larger than
$e^{R_{\gamma}(x)}$, which leads to the nearly exponential
asymptotics for $\Pr[N>n]$ in the following proposition and
Theorem~\ref{theorem:nearly}.

\begin{proposition}\label{theorem:nearlyB}
If $\log (\bar{F}^{-1}(x)) \sim e^{ \left( \log( \bar{G}^{-1}(x))
\right)^{\delta}}$, $\delta>1$, then,
 \begin{equation}
\log \left(\log \left(\Pr[N>n]^{-1} \right) \right)- \log n +(\log
n)^{\frac{1}{\delta}} \sim \frac{1}{\delta}\left( \log n
\right)^{\frac{2}{\delta}-1}. \nonumber
 \end{equation}
\end{proposition}
\begin{remark}
Observe that $\delta=2$ represents another critical point since
$\left( \log n \right)^{2/\delta-1}$ converges to $0$ or $\infty$ if
$\delta>2$ or $1<\delta<2$, respectively. Furthermore, the result
shows that $\Pr[N>n] \approx \exp \left(-n/e^{(\log
n)^{1/\delta}}\right)$, which means that $N$ is nearly exponential
because $e^{ \left( \log n \right)^{1/\delta}}$  is slowly varying
for $\delta>1$ (see p.~16 in \cite{BG87}). In addition, informally
speaking, we point out that the case $\delta= 1$ corresponds to the
Weibull case already covered by Theorem \ref{theorem:2I} in
Subsection \ref{ss:weibull}, meaning that this proposition describes
the change in functional behavior on the boundary between the
Weibull case and the nearly exponential one.
\end{remark}
\begin{proof}
First, we prove the \emph{upper bound}. Following the same approach
as in the proof of Theorem \ref{theorem:1I},  we obtain, for
$\epsilon>0$,
\begin{align}\label{eq:nearlyupperA1B}
  \Pr[N>n]  &\leq  \sum_{k=0}^{ n-1 } e^{-k- \frac{1}{1+\epsilon} e^{ \left(\log n - \log (k+1)\right)^{\delta}}}
+
  o\left( \Pr[N>n] \right).
\end{align}
Suppose that $f(x)\eqdef x+ \frac{1}{1+\epsilon} e^{ \left(\log n-
\log x\right)^{\delta}} $ reaches the minimum at $x^{\ast}$. It is
easy to see that $f'(x)=1-e^{(\log  n - \log
x)^{\delta}}/((1+\epsilon)x)$ is an increasing function in $x$ on
$(0,n)$. For $0<\epsilon<1$ define
\begin{align}
x_1&\eqdef \frac{n}{ e^{ \left( \log n - (1-\epsilon)(\log
n)^{1/\delta} \right)^{1/\delta} } }, \nonumber
\end{align}
 and for $n$ large enough, we obtain
 $$
f'(x_1)= 1- \frac{e^{-(1-\epsilon)(\log
n)^{1/\delta}}}{(1+\epsilon)} e^{\left(\log n - (1-\epsilon)(\log
n)^{1/\delta}\right)^{1/\delta}}\leq 1- \frac{e^{\epsilon (\log
n)^{1/\delta}- (1-\epsilon ^2) (\log
n)^{2/\delta-1}/\delta}}{1+\epsilon}<0,
$$
implying $f(x)'<0$ for $x<x_1$. Therefore, the minimum point
$x^{\ast}$ satisfies
\begin{align}\label{eq:nearlyupperA2B}
  x^{\ast} \geq x_1.
\end{align}
 Combining
(\ref{eq:nearlyupperA1B}) and  (\ref{eq:nearlyupperA2B}), we obtain,
for $n$ large,
\begin{align}
  \Pr[N>n]  & \leq  n e^{1-f(x^{\ast})} +
  o\left( \Pr[N>n] \right) \leq n e^{1 - x_1} +
  o\left( \Pr[N>n] \right)< 2n e^{1 - x_1},\nonumber
\end{align}
and therefore, for $n$ large enough,
\begin{equation}
 \log \left(\Pr[N>n]^{-1} \right)  \geq  \frac{n}{e^{(\log n - (1-\epsilon)( \log n)^{1/\delta})^{1/\delta}}}
  -
 \log (2n) -1,\nonumber
\end{equation}
which implies
\begin{equation}\label{eq:nearlyupperAB}
\log \left(\log \left(\Pr[N>n]^{-1} \right) \right)- \log n +(\log
n)^{\frac{1}{\delta}} \gtrsim  \frac{1}{\delta}\left( \log n
\right)^{\frac{2}{\delta}-1}.
\end{equation}

Next, we prove the \emph{lower bound}. By using the same arguments
as in the proof of the lower bound for Theorem \ref{theorem:2I}, we
obtain, for $n$ large enough,
\begin{align}
 \log \left(\Pr[N>n]^{-1} \right) &\leq  x_0 +\frac{1}{1-\epsilon} \log \left(\Phi\left(\frac{(1+\epsilon)n}{x_0}
  \right)\right)=x_0+\frac{1}{1-\epsilon}e^{\left( \log \left( \frac{(1+\epsilon)n}{x_0} \right)
  \right)^{\delta}}, \nonumber
\end{align}
   which, by choosing $x_0 = (1+\epsilon)ne^{-\left( \log n - (\log
   n)^{1/\delta}\right)^{1/\delta}}$,  passing $n\to \infty$ and then  $\delta \to 0$, yields
\begin{equation}\label{eq:nearlylowerAB}
\log \left(\log \left(\Pr[N>n]^{-1} \right) \right)- \log n +(\log
n)^{\frac{1}{\delta}} \lesssim  \frac{1}{\delta}\left( \log n
\right)^{\frac{2}{\delta}-1}.
\end{equation}

Finally, combining (\ref{eq:nearlyupperAB}) and
(\ref{eq:nearlylowerAB}) finishes the proof.
\end{proof}

\begin{theorem}\label{theorem:nearly}
If $\log (\bar{F}^{-1}(x)) \sim e^{R_{\gamma}\left(\bar{G}^{-1}(x)
\right)}$, where $R_{\gamma}(\cdot)$ is regularly varying with index
$\gamma>0$, then,
 \begin{equation}\label{eq:nearlyExp}
  \log \Pr[N>n]^{-1} \sim \frac{n}{ R_{\gamma}^{{\leftarrow}}(\log
  n)},
 \end{equation}
 where $R_{\gamma}^{{\leftarrow}}(\cdot)$ is the asymptotic inverse of
$R_{\gamma}(\cdot)$ as defined in Theorem 1.5.12 on p.~28 of
\cite{BG87}.
\end{theorem}
\begin{remark}
Note that the functional form in (\ref{eq:nearlyExp}) is different
from the one in (\ref{eq:NI2}) that describes the Weibull case. In
principle, one could study the situations when $\Phi(\cdot)$ grows
faster than three exponential scales, which would make the
distributions of $N$ even closer to the exponential one.  However,
from a practical point of view, these cases will basically be
indistinguishable from the exponential distribution and, thus, we
omit these derivations.
\end{remark}
\begin{proof}
First, we prove the \emph{upper bound}. Following the same approach
as in the proof of Theorem \ref{theorem:1I},  we obtain, for
$0<\epsilon<1$ and $y>0$,
\begin{align}\label{eq:nearlyupperA1}
  \Pr[N>n]  &\leq  \sum_{k=0}^{ \lfloor n/y \rfloor -1 } e^{-k- \frac{1}{1+\epsilon} e^{R_{\gamma} \left(\frac{n}{k+1}\right)}}
+
  o\left( \Pr[N>n] \right).
\end{align}
By using the same argument as in (\ref{eq:modified}), we can choose
$R^{\ast}_{\gamma}(x)=\gamma \int^{x}_{1}R(s)s^{-1}ds, x\geq 1$ and
$R^{\ast}_{\gamma}(\cdot)$ is absolutely continuous,
 strictly increasing with an inverse $R_{\gamma}^{\leftarrow}(\cdot)$.
Theorem 1.5.12 on p.~28 and Proposition 1.5.14 on p.~29 of
\cite{BG87} implies that
 $R_{\gamma}^{\leftarrow}(\cdot)$ is regularly varying with index $1/\gamma$ and is also the asymptotic inverse of
 $R_{\gamma}(\cdot)$.
Therefore,  there exists $y>0$ such that for $0<x<n/y$,
$$x+ \frac{1}{1+\epsilon} e^{R_{\gamma}\left(\frac{n}{x}\right)}
 \geq x+ \frac{1}{1+\epsilon}
 e^{(1-\epsilon)R^{\ast}_{\gamma}\left(\frac{n}{x}\right)}.
 $$

Suppose that $f(x)\eqdef x+  e^{(1-\epsilon)
R^{\ast}_{\gamma}\left(\frac{n}{x}\right)}/(1+\epsilon) $  reaches
the minimum at $x^{\ast}$, and note that
$$
 f'(x) = 1- \frac{1-\epsilon}{1+\epsilon}
 e^{(1-\epsilon)R^{\ast}_{\gamma}\left(\frac{n}{x}\right)} \left(R^{\ast}_{\gamma}\left(\frac{n}{x}\right) \right)'
     \frac{n}{x^2}=1- \frac{1-\epsilon}{1+\epsilon}e^{(1-\epsilon)R^{\ast}_{\gamma}\left(\frac{n}{x}\right)}
      \frac{\gamma R^{\ast}_{\gamma}\left(\frac{n}{x}\right)}{x}
$$
is an increasing function for $x$ on $(0,n/y)$. Now, defining
\begin{align}
x_1 &\eqdef \frac{n}{R_{\gamma}^{\leftarrow}\left(
\frac{1}{1-\epsilon} \left( \log n -
(1-\epsilon)\left(1+\frac{1}{\gamma}\right) \log \log n
 \right) \right)},
\end{align}
 it is easy to check that, for all $n$ large enough, $f'(x_1)$ is
equal to
$$
1-  \frac{\gamma R^{\leftarrow}_{\gamma}\left(\frac{1}{1-\epsilon}
\left(\log n - (1-\epsilon)\left(1+\frac{1}{\gamma}\right)\log \log
n \right) \right)   \left(\log n -
(1-\epsilon)\left(1+\frac{1}{\gamma}\right)\log \log n \right)
}{(1+\epsilon) (\log n)^{(1-\epsilon)(1+\frac{1}{\gamma})}}<0,
$$
which implies that $f(x)'<0$ for $0<x<x_1$ and $n$ large. Thus, the
minimum point $x^{\ast}$ satisfies
\begin{align}\label{eq:nearlyupperA2}
  x^{\ast}> x_1.
\end{align}
 Combining (\ref{eq:nearlyupperA1}) and
(\ref{eq:nearlyupperA2}) yields, for $n$ large enough,
\begin{align}
  \Pr[N>n]  & \leq  \frac{n}{y} e^{1-f(x^{\ast})} +
  o\left( \Pr[N>n] \right) \leq \frac{2n}{y} e^{1 - x_1},\nonumber
\end{align}
resulting in
\begin{equation}
 \log \left(\Pr[N>n]^{-1}\right)\geq \frac{n}{R_{\gamma}^{\leftarrow}\left(
\frac{1}{1-\epsilon} \left( \log n -
(1-\epsilon)\left(1+\frac{1}{\gamma}\right) \log \log n
 \right) \right)} -\log
  \left( \frac{2n}{y} \right)-1.\nonumber
  \end{equation}
  Therefore,  passing $n\to \infty$ and then $\epsilon \to 0$ in
  the preceding inequality yields
\begin{equation}\label{eq:nearlyupperA}
 \log \Pr[N>n]^{-1}  \gtrsim  \frac{n}{R_{\gamma}^{\leftarrow}(\log n)}.
\end{equation}

Next, we prove the \emph{lower bound}. By using the same arguments
as in the proof of the lower bound for Theorem \ref{theorem:2I}, we
obtain, for $n$ large enough,
\begin{align}
 \log \left(\Pr[N>n]^{-1} \right) &\leq  x_0 +\frac{1}{1-\epsilon} \log \left(\Phi\left(\frac{(1+\epsilon)n}{x_0}
  \right)\right)=x_0+\frac{1}{1-\epsilon}e^{ R_{\gamma} \left( \frac{(1+\epsilon)n}{x_0} \right)}, \nonumber
\end{align}
 which, by choosing $$
x_0= \frac{(1+\epsilon)n}{R_{\gamma}^{\leftarrow}\left( (1-\epsilon)
\log n -\frac{1}{\gamma} \log \log n
 \right)},
$$
and noting that $R_{\gamma}\left(
R_{\gamma}^{\leftarrow}(x)\right)\leq x/(1-\epsilon)$ for all $x$
large enough, yields, for $n$ large,
\begin{equation}
 \log \left(\Pr[N>n]^{-1} \right)
  \leq x_0+ \frac{n}{(1-\epsilon)(\log
  n)^{\frac{1}{(1-\epsilon)\gamma}
  }}.\nonumber
\end{equation}
The preceding inequality implies
\begin{equation}\label{eq:nearlylowerA}
 \log \left(\Pr[N>n]^{-1} \right)  \lesssim \frac{n}{R^{\leftarrow}_{\gamma}(\log n)}.
\end{equation}

Finally, combining (\ref{eq:nearlyupperA}) and
(\ref{eq:nearlylowerA}) finishes the proof. \qed
\end{proof}

\subsection{Asymptotics of the Total Transmission Time $T$}\label{ss:asmpT}
In this subsection, we compute the asymptotics of the total
transmission time $T$ based on the previous results on $\Pr[N>n]$.
Our proving technique involves the relationship between $N$ and $T$
described in (\ref{defeq:T}) and the classical large deviation
results. Theorem \ref{theorem:asympT} and Theorem \ref{theorem:T}
characterize the exact asymptotics and logarithmic asymptotics for
the very heavy case, respectively, and Theorem
\ref{theorem:TWeibull} derives the result for the moderate heavy
(Weibull) case. Interestingly, we want to point out that, unlike
Theorems \ref{theorem:asympT} and \ref{theorem:T} requiring no
conditions on $A$ (Theorem \ref{theorem:asympT} needs
$\expect[A]<\infty$), the minimum conditions needed for Theorem
\ref{theorem:TWeibull}, as shown by Proposition
\ref{prop:WeibullBalance},  basically involve a balance  between the
tail decays of $\Pr[A>x]$ and $\Pr[L>x]$.

 Similarly, the corresponding results on $T$ can be derived
for the other statements on $N$, e.g., Propositions
\ref{eq:proposition1}, \ref{eq:proposition2}, \ref{prop:gtrOne},
\ref{prop:betweenHalfOne}, and Theorem \ref{theorem:nearly}. But, to
avoid lengthy expositions and repetitions, we omit this derivations.
In the following, let $\vee \equiv \max$.
\begin{theorem} \label{theorem:asympT}
If $\expect\left[U^{(\alpha \vee 1)+\theta}\right]<\infty$, $\expect
\left[A^{1+\theta} \right]<\infty$ and $\expect\left[L^{\alpha
+\theta} \right]<\infty$ for some $\theta>0$, then, under the same
conditions as in Theorem \ref{theorem:asympN2} i), i.e.,
$\bar{F}^{-1}(x) \thicksim \Phi\left( \bar{G}^{-1}(x)\right)$ with
$\Phi(x)$ being regularly varying of index $\alpha>0$, we obtain, as
$t\rightarrow \infty$,
 \begin{equation}\label{eq:asympT}
\Pr[T>t]\thicksim\frac{\Gamma( \alpha+1 )(\expect[U+A])^{\alpha}}{
\Phi(t)}.
\end{equation}
\end{theorem}
\begin{remark}
Note that $\expect \left[L^{\alpha+\theta} \right]<\infty$ is
basically a minimum condition for $\alpha>1$ since $\expect
\left[L^{\alpha-\theta} \right]=\infty$ implies $\expect
\left[T^{\alpha-\theta} \right]=\infty$ because of $T\geq L$, which
would contradict (\ref{eq:asympT}).
\end{remark}

The \textbf{proof} is presented in Subsection \ref{ss:asympT}.

\begin{theorem}\label{theorem:T}
Under the same conditions of Theorem \ref{theorem:1I}, i.e.,
  the eventually non-decreasing function $\Phi(x)\eqdef e^{l(x)}$ satisfies (\ref{eq:conditionI2})
  where $l(x)$ is slowly varying with
 \begin{equation}\label{eq:lxB}
  \lim_{x \to
  \infty} \frac{l\left( \frac{x}{l(x)} \right)}{l(x)} = 1,
  \end{equation}
and in addition, if $\Pr[L>x]=O\left(\Phi(x)^{-(\delta+1)}\right)$
and $\Pr[ U> x ]=O\left( \Phi(x)^{-\left( \delta+1\right)} \right)
$, $\delta>0$, then, we obtain
  \begin{equation}\label{eq:T}
 \lim_{t \rightarrow \infty}\frac{\log \left(\Pr[T>t]^{-1}\right)}{ \log(\Phi
   (t))}=1.
   \end{equation}
\end{theorem}

\begin{remark}
This result implies parts (1:1), (2:1) and (2:2) of Theorem 2.1 in
\cite{Asmussen07} and extends Theorem 2 in \cite{PJ07RETRANS}.
Furthermore, it shows that, if $\log \Pr[L>x]^{-1} \approx \alpha
\log\Pr[A>x]^{-1}$, meaning that the hazard functions of $L$ and $A$
are asymptotically linear,  the distribution tails of the number of
transmissions and total transmission time are essentially power
laws. Thus, the system can exhibit high variations and possible
instability, e.g., when $0<\alpha<2$, the transmission time has an
infinite variance and, when $0<\alpha<1$, it does not even have a
finite mean.
\end{remark}

\begin{remark}
It is easy to understand that, if the data sizes (e.g., files,
packets) follow heavy-tailed distributions, the total transmission
time is also heavy-tailed. However, from these two theorems, we see
that even if the distributions of the data and channel
characteristics are highly concentrated, e.g.,  when they are
asymptotically proportional on the logarithmic scale (see
Corollary~\ref{co:normal} in Subsection \ref{ss:VHA}), the
heavy-tailed transmission delays can still arise.
\end{remark}
The \textbf{proof} is presented in Subsection \ref{ss:ProofSlowT}.

\begin{theorem}\label{theorem:TWeibull}
Under the same conditions of Theorem \ref{theorem:2I}, i.e.,
 the eventually non-decreasing function $\Phi(x)\eqdef e^{R_{\beta}(x)}$ satisfies (\ref{eq:conditionI2})
    where $R_{\beta}(x)=x^{\beta}l(x)$, $\beta>0$
   is regularly varying with $l(x)$ satisfying
\begin{equation}\label{eq:lxC}
  \lim_{x \to
  \infty} \frac{l\left( \left( \frac{x}{l(x)} \right)^{\frac{1}{1+\beta}} \right)}{l(x)} = 1,
  \end{equation}
and in addition, if $\expect[A]<\infty$,
 $\Pr[U> x]=O \left( e^{-
(\log \Phi(x))^{(1+\delta)/(\beta+1)}} \right)$, $\delta>0$,  and
$\Pr[L>x]=O\left( e^{- x^{\xi}} \right)$, $\Pr[A>x]=O\left( e^{-
x^{\zeta}} \right)$ with $\xi>\beta/(\beta+1), \zeta \geq 0$
satisfying $(1-\zeta)\beta < \xi$, then, we obtain
   \begin{equation}\label{eq:TWeibull}
   \lim_{n \rightarrow \infty}\frac{ \log \left(\Pr[T>t]^{-1}\right)}{ \left(\log\Phi(
   t) \right)^{\frac{1}{\beta+1}}}= \frac{ \beta^{\frac{1}{\beta+1}} +
\beta^{-\frac{\beta}{\beta+1}}
}{(\expect[A+U])^{\frac{\beta}{\beta+1}}}.
   \end{equation}
\end{theorem}

\begin{remark}
This theorem  implies part (1:2) of Theorem 2.1 in
\cite{Asmussen07}, and provides a more precise logarithmic
asymptotics instead of a double logarithmic limit. Furthermore,  it
is easy to check that the condition $(1-\zeta)\beta < \xi$ holds in
two special cases: (i) if $\zeta \geq \beta/(\beta+1)$ and $\xi
>\beta/(\beta+1)$ or (ii) if $\xi>\beta$ and $\zeta=0$ (assuming no conditions for
$\Pr[A>x]$ beyond $\expect[A]<\infty$).
\end{remark}

The \textbf{proof} is presented in Subsection \ref{ss:ProofWeibull}.
 Basically, the condition $(1-\zeta)\beta < \xi$ (or
equivalently $\xi/(\xi+1-\zeta)>\beta/(\beta+1)$) is needed since
the following proposition shows that $\Pr[T>t]$ could have a heavier
tail than predicted by (\ref{eq:TWeibull}) if $(1-\zeta)\beta >
\xi$.

\begin{proposition}\label{prop:WeibullBalance}
 If $\Pr[L>x]= e^{- x^{\xi}} $ and $\Pr[A>x]= e^{-
x^{\zeta}}$ with $0<\xi, \zeta <1$, then, as $t \to \infty $,
$$\Pr[T>t] \gtrsim e^{-2 t^{\xi/(\xi+1-\zeta)}}.$$
\end{proposition}

\begin{proof}
It is easy to see that, for $\delta, y>0$,
\begin{align}
\Pr\left[ T>t \right]&\geq \Pr\left[ T>t, y<A_i<(1+\delta)y, 1\leq
i\leq \frac{t}{y}, L>(1+\delta)y \right] \nonumber\\
&\geq
\left(\Pr[y<A<(1+\delta)y]\right)^{\frac{t}{y}}\Pr[L>(1+\delta)y],
\nonumber
\end{align}
which, by noting that $\Pr[A>x]= e^{- x^{\zeta}}$ with $\zeta>0$,
yields
\begin{align}
\Pr\left[ T>t \right] &\gtrsim
\left(\Pr[A>y]\right)^{\frac{t}{y}}\Pr[L>(1+\delta)y] =e^{-\left(
\frac{t}{y} y^{\zeta} +y^{\xi}\right)}.
\end{align}
Choosing $y=t^{1/(\xi+1-\zeta)}$ finishes the proof. \qed
\end{proof}

\section{Engineering Implications}\label{s:APP}
As already stated in the introduction, retransmissions are the
integral component of many modern networking protocols on all
communication layers from the physical to the application one. In
our recent work
\cite{Jelen07e2e,PJ07RETRANS,Jelen07ALOHA,Jelen07TRe2e}, we have
shown that these protocols may result in heavy-tailed (e.g., power
law) delays even if all the system components are light-tailed
(superexponential). More specifically,  from an engineering
perspective, our main discovery is the matching between the
statistical characteristics of the channel and transmitted data
(e.g., packets). Basically, one can expect good or bad delay
performance measured by the existence of $\alpha$-moments for $N$
and $T$ if $\alpha \log \Pr[A>x]
> \log \Pr[L>x]$ or $\alpha  \log \Pr[A>x] < \log \Pr[L>x]$, respectively.
 Note that, if $\alpha<1$, then the system could experience zero
throughput.

On the network application layer, most of us have experienced the
connection failures while downloading a large file from the
Internet. This issue has been already recognized in practice where
software for downloading sizable documents was developed that would
save the intermediate data (checkpoints) and resume the download
from the point when the connection was broken. However, our results
emphasize that, in the presence of frequently failing connections,
the long delays may arise even when downloading relatively small
documents. Hence, we argue that one may need to adopt the
application layer software for the wireless environment by
introducing checkpoints even for small to moderate size documents.

Furthermore, on the physical layer, it is well known that wireless
links, especially for low-powered sensor networks, have higher error
rates than the wired counterparts. This may result in large delays
on the data link layer due to the (IP) packet variability and
channel failures. Therefore, our results suggest that packet
fragmentation techniques need to be applied with special care since:
if the packets are too small, they will mostly contain the packet
header, which can limit
 the useful throughput; if the packets are too large, power law delays can
deteriorate the quality of transmission. When the codewords, the
basic units of packets in the physical layer, are much smaller than
the maximum size of the packets, our results show that the number of
retransmissions could be power law, which challenges the traditional
model that assumes a geometric number of retransmissions. We believe
that short codewords are realistic assumption for sensor networks,
where complicated coding schemes are unlikely since the nodes have
very limited computational power. In reality, packet sizes may have
an upper limit (e.g., WaveLAN's maximum transfer unit is $1500$
bytes), this situation may result in truncated power law
distributions for $T$ and $N$ in the main body with a stretched
(exponentiated) support in relation to the support of $L$ (see
Example~$3$ in Section~IV of \cite{PJ07RETRANS}) and, thus, may
result in very long, although, exponentially bounded delays. The
impact of truncated heavy-tailed distributions on queueing behavior
was quantified in \cite{Je99}.

On the medium access control layer, ALOHA is a widely used protocol
that provides a contention management scheme for multiple users
sharing the same medium. Once a user detects a collision, it will
back off for a random (exponential) period of time before trying to
retransmit the collided packet. Due to its simplicity and
distributed nature, ALOHA is the basis of many other protocols, such
as CSMA/CD. We discovered a new phenomenon in \cite{Jelen07ALOHA}
that a basic finite population ALOHA model with variable size
(exponential) packets is characterized by power law transmission
delays, possibly even resulting in zero throughput; see Theorem 1
and Example 1 in \cite{Jelen07ALOHA} that characterizes and
illustrates the observation respectively. This power law effect
might be diminished, or perhaps eliminated, by reducing the
variability of packets. However, we also show in \cite{Jelen07ALOHA}
that even a slotted (synchronized) ALOHA with packets of constant
size can exhibit power law delays when the number of active users is
random; see Theorem 2 and Example 2 in \cite{Jelen07ALOHA}. The
ALOHA system is a generalization of our study in this paper, since,
informally, it can be viewed as the state dependent version of the
model considered here where the distributions of $L$ and $A$ depend
on the state of the system.

On the transport layer, most of the network protocols (e.g., TCP)
use end-to-end acknowledgements for packets as an error control
strategy. Namely, once the packet sent from the sender to the
receiver is lost due to, e.g., finite buffers or link failures, this
packet will be
retransmitted by the sender. 
Furthermore, the number of hops that a packet traverses on its path
to the destination is random, e.g., an end-user that is surfing the
Web might download documents from diverse web sites. Our recent work
in \cite{Jelen07e2e} shows that this acknowledgement mechanism,
jointly with the random number of hops, may result in heavy-tailed
(e.g., power law) delays. To illustrate this phenomenon, we consider
the following basic example. Assume that a single data unit (packet)
needs to traverse a random geometric number $L$ of hops before
reaching the destination, $\Pr[L>n]=e^{-p n}, p>0$. Next, assume
that in each hop the packet can independently (independent of $L$ as
well) be lost with probability $1-e^{-q}, q>0$. When a packet is
lost, it is retransmitted by the sender and this procedure continues
until the packet reaches its destination. Then, it is easy to see
that, in conjunction with Remark \ref{remark:lattice} after Theorem
\ref{theorem:asympN2}, the number $N$ of retransmissions that the
sender needs to perform satisfies
$$
 \frac{e^{-p} \Gamma(1+p/q) }{n^{p/q}} \leq \Pr[N>n] \leq  \frac{e^{p}\Gamma(1+p/q)}{n^{p/q}}.
$$
Similarly, assuming that in each hop a packet is processed for one
unit of time, we can derive that the distribution of the total
transmission time $T$ satisfies $\log (\Pr[T>t]) \sim - (p/q) \log
t$.

 Furthermore, when the cause of losses is due to the finiteness of
 buffers, i.e., a packet is lost when it sees a full buffer upon its
 arrival,  the preceding general setup can be more precisely modeled as a sequence of random number $L$ of
 tandem queues \cite{Jelen07e2e}. More specifically, we consider $L$ tandem $\cdot/M/1/b$
queues with each queue being able to accommodate up to (finitely
many) $b$ packets; $M$ stands for exponential (memoryless) service
times. This model can be shown to result in heavy-tailed delays
under quite general assumptions on cross traffic, network topology
and routing scenarios. However, for
 simplicity we only present the following example. As
depicted in Figure \ref{fig:tandemQueueLoop}, suppose that the
single packet, as well as the cross traffic, is sent sequentially
through a chain of finite buffer queues with capacity $b$. Also, we
assume that the cross traffic flows are i.i.d Poisson processes and
 the  service requirements needed for processing different packets
and the same packet at different queues are i.i.d. exponential
random variables. Furthermore, the sender tries to transmit a single
packet through the sequence of queues, and if the packet is lost,
the packet will be immediately retransmitted by the sender. Then,
regardless of how many hops the cross traffic flows traverse before
leaving the system, the distribution of the number of
retransmissions $N$ satisfies the following Theorem
\ref{theorem:tamdem}.
 \begin{figure}[h]
\centering
\includegraphics[width=15cm]{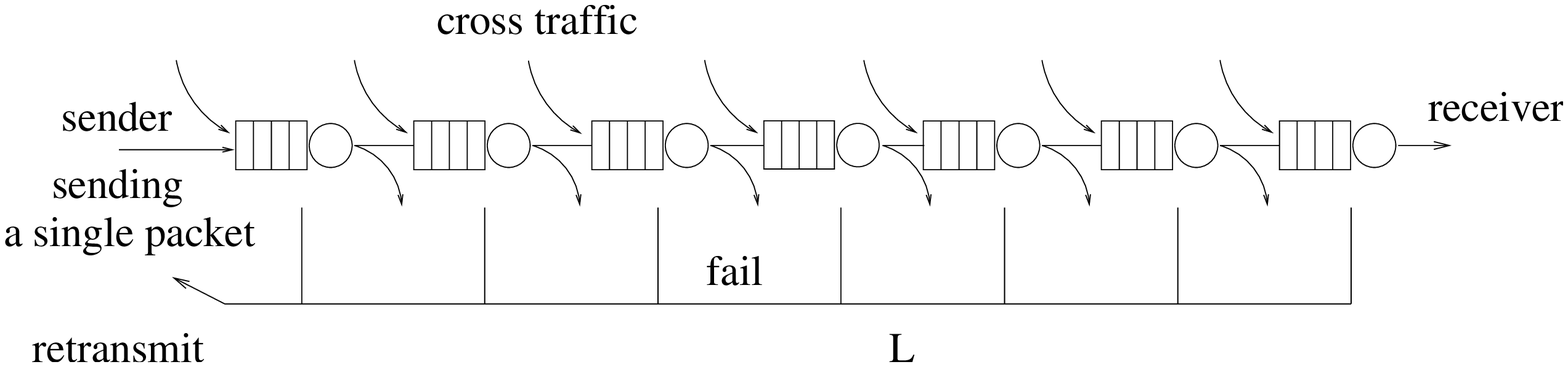}
  \caption{Tandem $\cdot/M/1/b$ queues with finite buffers}
\label{fig:tandemQueueLoop}
\end{figure}
\begin{theorem}\label{theorem:tamdem}
 If the limit $
   p\eqdef \lim_{n\rightarrow \infty} \log \left(\Pr[L>n] \right)/n<0$ exists, then, there exists $0< \alpha_1 \leq
\alpha_2<\infty$, such that
   \begin{align*}
 -\alpha_2 & \leq  \varliminf_{n \rightarrow \infty}\frac{\log \Pr[N>n]}{\log
   n}\leq \varlimsup_{n \rightarrow \infty}\frac{\log \Pr[N>n]}{\log
   n} \leq -\alpha_1.
   \end{align*}
\end{theorem}
\begin{remark}
Note that the same result can be easily derived for $T$.
Furthermore, due to the generality of our argument, the proof of
this theorem can be applied to much more complicated routing
schemes, network topologies and cross traffic conditions, e.g., with
routing loops. However, such generalizations, except for complex
notation, do not bring new insights, and therefore, we only study
the current simple example. Further study of this model will be
available in \cite{Jelen07TRe2e}.
\end{remark}
\begin{proof}
Number the sequence of queues from $1$ to $L$ sequentially. Recall
that the packet is lost when it sees a full buffer upon its arrival.
In order to prove the theorem, we construct two systems with
independent loss probabilities at different queues that provide
upper and lower bounds on the loss probabilities for the considered
packet.

  First, we prove the \emph{lower bound}.
 Construct a system that empties queue
$i+1$ whenever the considered packet begins receiving service in
queue $i$ ($1\leq i< L$).  Denote by $C_i$ the event that this
packet is lost when arriving at queue $i+1$. From the procedure of
this construction and the memoryless property, it is easy to see
that $\{C_i\}_{1\leq i\leq L}$ are i.i.d. conditional on $L$. Thus,
we obtain, for $n_0>0$,
\begin{align}\label{eq:tamdemlow1}
 \Pr[N>n] & \geq  \expect \left[\left(1- \prod_{i=1}^{L}(1-\Pr[C_i]) \right)^n {\Bigg | } L
 \right]\geq
 \expect\left[\left(1- (1-\Pr[C_1])^L \right)^n \ind(L>n_0)\right].
\end{align}
Then, we construct a continuous random variable $L^{\ast}$ with
$\Pr[L^{\ast}>x]=e^{-2px}, x\geq 0$,
and  choose $n_0$ large enough such that
 $\Pr[L^{\ast}>x]\leq \Pr[L>x]$ for all $x>n_0$.  Therefore,
 by using stochastic dominance and
replacing $L$ with $L^{\ast}$ in (\ref{eq:tamdemlow1}), we obtain
\begin{align}
 \Pr[N>n] &
 \geq \expect\left[\left(1- (1-\Pr[C_1])^{L^{\ast}} \right)^n \right]-\left(1- (1-\Pr[C_1])^{n_0}
 \right)^n,
\end{align}
which, by setting $\bar{G}(x)=(1-\Pr[C_1])^x$,
$\bar{F}(x)=\Pr[L^{\ast}>x]$ and applying Theorem \ref{theorem:1I},
yields, for some $\alpha_2>0$,
 \begin{align}\label{eq:tamdemlow}
 \varliminf_{n \rightarrow \infty}\frac{\log \Pr[N>n]}{\log
   n} &\geq -\alpha_2.
   \end{align}
Note that this line of argument can be used to rigorously prove
Remark \ref{remark:lattice2}.

 Next, we prove the \emph{upper bound}.
 Construct a system that empties queue $i+1$ whenever the considered packet
begins receiving service in queue $i$ ($1\leq i< L$). Then, using
this construction and similar arguments as in the proof of the lower
bound, we can easily prove that there exists $\alpha_1>0$ such that
\begin{align}
 \varlimsup_{n \rightarrow \infty}\frac{\log \Pr[N>n]}{\log
   n} &\leq -\alpha_1, \nonumber
   \end{align}
which, in conjunction with (\ref{eq:tamdemlow}), finishes the proof.
\qed
\end{proof}

Finally, we would like to point out that, in addition to the
preceding applications in communication networks
\cite{Jelen07e2e,PJ07RETRANS,Jelen07ALOHA,Jelen07TRe2e} and job
processing on machines with failures \cite{FS05,SL06}, the model
studied in this paper may represent a basis for understanding more
complex failure prone systems, e.g., see the recent study on
parallel computing in \cite{Asmussen07July}.

 In conclusion, we would like to
emphasize that, in practice, our results provide an easily
computable benchmark for measuring the tradeoff between the data
statistics and channel characteristics that permits/prevents
satisfactory transmission.

%
%

\section{Proofs}\label{s:proofs}
\subsection{Proof of Proposition \ref{prop:sub}}\label{ss:sub}
As stated earlier in Subsection \ref{s:model}, the proof of this
proposition was originally presented in Lemma 1 of
\cite{PJ07RETRANS} and, we repeat it here for reasons of
completeness.

 \begin{proof}
Note that for any $\delta>0$, there exists $t_{\delta}>0$ such that,
 for all $0<t<t_{\delta}$,
 $$
 1- t\geq e^{-\delta} e^{-t}.
$$
 Therefore, we can choose $x_{\delta}$ large enough, such that
$1-\bar{G}(x)\geq e^{-\delta}e^{-\bar{G}(x)}$ for all
$x>x_{\delta}$. Then,
\begin{align*}
e^{\epsilon n}\Pr[N>n] &\geq e^{\epsilon n}\expect \left[ \left( 1-
\bar{G}(L)\right)^n \ind(L\geq x_{\delta}) \right]
\geq e^{\epsilon n}\expect \left[ e^{n \delta}e^{-n\bar{G}(L)}\ind(L\geq x_{\delta}) \right]\\
&\geq \left( e^{\epsilon-\bar{G}(x_{\delta})-\delta}\right)^n
\bar{F}( x_{\delta}).
\end{align*}
Thus, by selecting $\delta$ small enough and $x_{\delta}$ large
enough, we can always make $e^{\epsilon-\bar{G}(x_{\delta})-\delta}
>1$, and, by passing $n\rightarrow \infty$, we complete
the proof of (\ref{eq:subN}).

Next, we prove the corresponding result for $T$.  Suppose that
$\bar{G}(x_0)>0$ for some $x_0>0$; otherwise, $T$ will be infinite,
which yields (\ref{eq:subT}) immediately. We can always find
$x_1>x_0>0$, such that i.i.d. random variables $X_i \eqdef x_0
\ind(x_0<A_i<x_1)$ satisfy $0<\expect X_1<\infty$. Now, for any
$\zeta>0$,
\begin{align}\label{eq:I1I2}
  \Pr[T>t] &= \Pr\left[\sum_{i=1}^{N-1}(U_i+A_i)+L>t\right]
  \geq \Pr\left[\sum_{i=1}^{N-1}A_i \ind(x_0<A_i<x_1)>t\right] \nonumber\\
  &\geq \Pr\left[\sum_{i=1}^{N-1}X_i>t\right]
  \geq \Pr\left[\sum_{i=1}^{N-1}X_i>t, N\geq \frac{t(1+\zeta)}{\expect X_1}\right]\nonumber\\
  &\geq \Pr\left[N > \frac{t(1+\zeta)}{\expect X_1}+1\right]  -\Pr\left[\sum_{i=1}^{N-1}X_i \leq t, N>
   \frac{t(1+\zeta)}{\expect X_1}+1\right] \nonumber\\
   &\eqdef I_1-I_2.
\end{align}
Since, for $\bar{X}_i\eqdef \expect[X_i]-X_i$,
\begin{align}\label{eq:known}
 I_2 &\leq \Pr\left[\sum_{i \leq t(1+\zeta)/\expect X_1} X_i\leq t
 \right]=\Pr\left[\sum_{i \leq t(1+\zeta)/\expect X_1} \bar{X}_i \geq \zeta
 t\right],
\end{align}
it is well known (e.g., see Example 1.15 of \cite{SW95}) that there
exists $\eta>0$, such that
\begin{equation}\label{eq:subI1}
I_2\leq e^{-\eta t}.
\end{equation}
Therefore, by (\ref{eq:subN}), (\ref{eq:I1I2}) and (\ref{eq:subI1}),
we obtain that for all $0<\epsilon<\eta$,
\[
 e^{\epsilon t}\Pr[T>t] \rightarrow \infty \;\; \text{as}\; t
        \rightarrow \infty,
\]
implying that (\ref{eq:subT}) holds for any $\epsilon>0$.  \qed
 \end{proof}

\subsection{Proof of Proposition \ref{prop:slowVarying}}\label{ss:slowVarying}
\begin{proof}
If $\log (\Phi(x))$ is slowly varying, then, for any $0<\delta<
\epsilon$, there exists $x_{\delta}>0$ such that $\log (\Phi(x))<
x^{\delta}$ for all $x>x_{\delta}$.
 By using the condition (\ref{eq:conditionI2}),  or equivalently (\ref{eq:equivBound2I}),  we
 obtain,  for $n$ large enough,
\begin{align}
 \Pr[N>n]& = \expect[(1-\bar{G}(L))^n] \nonumber\\
 &\geq  \left( 1- \frac{1}{n} \right)^n \Pr \left[\bar{G}(L) \leq \frac{1}{n}
 \right] \nonumber\\
 &\geq  \left( 1- \frac{1}{n} \right)^n \Pr \left[\Phi^{\leftarrow} \left( \bar{F}^{-(1+\epsilon)}(L)
\right) \geq  n
 \right] \nonumber\\
 &\geq \left( 1- \frac{1}{n} \right)^n e^{-x^{\delta}},
 \nonumber
\end{align}
where we use the fact that for $x_{\epsilon}$ chosen in
(\ref{eq:equivBound2I}) one can always select $n$ large enough such
that $\{\bar{G}(L)\leq 1/n \} \subset \{ L>x_{\epsilon}\}$.
Therefore, we obtain,
\begin{align}
 0\leq \varlimsup_{n \to \infty} \frac{ -\log \Pr[N>n]}{ n^{\epsilon}} &  \leq \lim_{n \to \infty}
     \frac{-1+ n^{\delta}}{n^{\epsilon}}=0, \nonumber
\end{align}
which proves the proposition. \qed
\end{proof}

\subsection{Continuation of the proof of Theorem
\ref{theorem:asympN2}}\label{ss:continuation}
\begin{proof}
Now, we prove the \emph{lower bound}. For $K>0$ and $x_{\epsilon}$
selected in (\ref{eq:equivBound2}), choosing $x_{n}>x_{\epsilon}$
 with $ \Phi^{\leftarrow}\left(
 (1-\epsilon)\bar{F}(x_{n})
\right) = n/K$, we obtain, for large $n$,
\begin{align}
\Pr[N>n] &=  \expect \left[ \left(1-\bar{G}(L) \right)^n \right]\nonumber\\
&\geq  \expect \left[ \left(1-\bar{G}(L) \right)^n \ind(  L
> x_{n}) \right]\nonumber\\
 &\geq \expect \left[
\left(1- \frac{1}{\Phi^{\leftarrow}\left((1-\epsilon)V^{-1}\right)}
 \right)^n  \ind( V< \bar{F}(x_{n})) \right ],\nonumber
\end{align}
which, by letting $z=n/\Phi^{\leftarrow}\left( (1-\epsilon) V^{-1}
\right)$, yields
\begin{align}\label{eq:asympN6}
\Pr[N>n]\Phi(n)
        & \geq \int_{0}^{K}\left(1- \frac{z}{n}
\right)^n\frac{\Phi(n)}{\Phi\left( n/z \right)}\frac{ \Phi'\left(
n/z \right)}{\Phi\left( n/z \right)}
        \frac{(1-\epsilon)n}{z^2}dz.
\end{align}
From (\ref{eq:asympN6}), by using the same approach as in deriving
(\ref{eq:propupper5I1}), we obtain, as $n \to \infty$,
\begin{equation}
\Pr[N>n]\Phi(n) \sim \int_{0}^{K}(1-\epsilon)\alpha e^{-z}
z^{\alpha-1}dz, \nonumber
\end{equation}
which, by passing  $K \to \infty$ and $\epsilon \to 0$,  yields
  \begin{equation}\label{eq:asympN7}
   \Pr[N>n]\Phi(n) \gtrsim  \int_{0}^{\infty}\alpha e^{-z}
z^{\alpha-1}dz=\Gamma(\alpha+1).
\end{equation}
Combining (\ref{eq:asympN5}) and (\ref{eq:asympN7}) completes the
proof of (\ref{eq:asympNa}).

Then, we proceed with proving (\ref{eq:asympNb}).  First, we prove
the \emph{lower bound}. Since $\Phi(x)$ is eventually
non-decreasing, we obtain the inequality presented in
(\ref{eq:equivBound2}) again,
and therefore, for $n$ large enough and $\epsilon>0$,
\begin{align}
 \Pr[N>n]& = \expect[(1-\bar{G}(L))^n] \nonumber\\
 &\geq  \left( 1- \frac{\epsilon}{n} \right)^n \Pr \left[\bar{G}(L) \leq \frac{\epsilon}{n}
 \right] \nonumber\\
 &\geq  \left( 1- \frac{\epsilon}{n} \right)^n \Pr \left[\Phi^{\leftarrow} \left((1-\epsilon) \bar{F}^{-1}(L)
\right) \geq  \frac{n}{\epsilon}
 \right] \nonumber\\
 &\geq \left( 1- \frac{\epsilon}{n} \right)^n \frac{1-\epsilon}{\Phi \left(\frac{n}{e}\right)} \; ,
 \nonumber
\end{align}
implying
\begin{align}
 \varliminf_{n \to \infty} \Pr[N>n] \Phi(n)  &\geq  \varliminf_{n \to \infty}\left( 1- \frac{\epsilon}{n} \right)^n
 \frac{(1-\epsilon)\Phi(n)}{\Phi \left(\frac{n}{\epsilon}\right)},\nonumber
\end{align}
which, by passing $\epsilon \to 0$, yields
\begin{equation}\label{eq:lowerlb}
 \varliminf_{n \to \infty} \Pr[N>n] \Phi(n)  \geq  1.
\end{equation}

Next, we prove the \emph{upper bound}. Using a similar approach that
derived (\ref{eq:propupper1}), we obtain
\begin{align}
  \Pr[N>n]  &\leq \expect \left[  e^{- \frac{n}{\Phi^{\leftarrow}\left((1+\epsilon)V^{-1}\right)} } \right] +
  \left(1-\bar{G}(x_{\epsilon}) \right)^n  \nonumber\\
  & \leq \Pr\left[0\leq  \frac{n}{\Phi^{\leftarrow}\left((1+\epsilon)V^{-1}\right)} \leq e^m \right]
      \nonumber\\
      &\;\;\;\;\; + \sum_{k=m}^{\lceil \log (\epsilon n) \rceil} e^{-e^{k}}
        \Pr\left[ e^k \leq  \frac{n}{\Phi^{\leftarrow}\left((1+\epsilon)V^{-1}\right)} \leq e^{k+1} \right]
             + o\left( \frac{1}{\Phi(n)} \right)  \nonumber\\
  &\leq  \frac{1+\epsilon}{\Phi\left(\frac{n}{e^m}\right)}
   + \sum_{k=m}^{\lceil \log (\epsilon n) \rceil} e^{-e^{k}}
          \frac{1+\epsilon}{\Phi \left(  \frac{n}{e^{k+1}} \right)}
      + o\left( \frac{1}{\Phi(n)} \right), \nonumber
\end{align}
resulting in
\begin{equation}\label{eq:lowerlc}
\Pr[N>n]\Phi(n) \leq
\frac{(1+\epsilon)\Phi(n)}{\Phi\left(\frac{n}{e^m}\right)}
   +  \sum_{k=m}^{\lceil \log (\epsilon n) \rceil} e^{-e^{k}}
          \frac{(1+\epsilon)\Phi(n)}{\Phi \left(  \frac{n}{e^{k+1}} \right)}
      + o\left( 1 \right).
\end{equation}
Note that the second term in the right hand side of
(\ref{eq:lowerlc}) is always finite because of (\ref{eq:prop1c})
and, by passing $n \to \infty$ and then $m\to \infty$ in
(\ref{eq:lowerlc}), we obtain
\begin{equation}\label{eq:uplb}
 \varlimsup_{n \to \infty}\Pr[N>n]\Phi(n) \leq 1.
\end{equation}
Combining (\ref{eq:lowerlb}) and (\ref{eq:uplb}) finishes the proof
of  (\ref{eq:asympNb}). \qed
\end{proof}

\subsection{Proof of Proposition \ref{prop:gtrOne}}\label{ss:gtrOne}
\begin{proof}
First, we prove the \emph{lower bound}. By recalling the condition
(\ref{eq:propo1A}), or equivalently (\ref{eq:equivBound2}), and
using $1-x\geq e^{-(1+\epsilon)x}$ for $x$ small enough, we obtain,
for $n$ large enough and $x_0>0$,
\begin{align}
  \Pr[N>n] &\geq \expect[(1-\bar{G}(L))^n
  \ind(L\geq x_{\epsilon})] \geq \expect[e^{-(1+\epsilon)\bar{G}(L)n}
  \ind(L\geq x_{\epsilon})] \nonumber\\
  &\geq \expect \left[  e^{- \frac{(1+\epsilon)n}{\Phi^{\leftarrow}\left((1-\epsilon)V^{-1}\right)} }
   \ind(V\leq \bar{F}(x_{\epsilon})) \right]  \geq  e^{-x_0}
        \Pr\left[  \frac{(1+\epsilon)n}{\Phi^{\leftarrow}\left((1-\epsilon)V^{-1}\right)} \leq x_0,
          V\leq \bar{F}(x_{\epsilon}) \right]
        \nonumber\\
  & = e^{-x_0 }(1-\epsilon)\Phi\left(\frac{(1+\epsilon)n}{x_0} \right)^{-1}
   = (1-\epsilon)e^{-x_0-\lambda\left(\log n - \log \left(\frac{x_0}{1+\epsilon}\right) \right)^{\delta}},
\end{align}
Using the preceding inequality and setting $x_0=\lambda \delta(\log
n)^{\delta-1}$ yields, for $n$ large enough,
\begin{align}
\log \Pr[N>n]^{-1} - \lambda(\log n)^{\delta} &\leq \lambda \left(
\log n - \log \left(\frac{x_0}{1+\epsilon}\right) \right)^{\delta} -
\lambda(\log n)^{\delta}+x_0 -\log
(1-\epsilon) \nonumber\\
&\leq  - (1-\epsilon)\lambda \delta \left( \log n \right)^{\delta-1}
\log \left( \lambda \delta(\log n)^{\delta-1} \right)   + \lambda
\delta(\log n)^{\delta-1}, \nonumber
\end{align}
which, by passing $n \to \infty$ and then $\epsilon \to 0$, results
in
\begin{equation}\label{eq:gtrOnelowerA}
 \log \Pr[N>n]^{-1} - \lambda(\log n)^{\delta} \lesssim -\lambda \delta(\delta-1) (\log \log n)(\log n)^{\delta-1}.
\end{equation}

Next, we prove the \emph{upper bound}. Following the same approach
as in the proof of Theorem \ref{theorem:1I}, we obtain, for large
$n$ and $y=\lambda(\log n)^{\delta} - \lambda \delta(\delta-1) \log
\log n(\log n)^{\delta-1}$,
\begin{align}\label{eq:gtrOneupper1}
  \Pr[N>n] &\leq \sum_{k=0}^{y-1} e^{-k}
        \Pr\left[ k \leq  \frac{n}{\Phi^{\leftarrow}\left((1+\epsilon)V^{-1}\right)} \leq k+1 \right]
        + e^{-y}+  o\left( \Pr[N>n] \right) \nonumber\\
  & \leq  (1+\epsilon)\sum_{k=0}^{ y-1 } e^{-k- \lambda \left( \log n - \log (k+1) \right)^{\delta}}
 + e^{-y}+
  o\left( \Pr[N>n] \right) .
\end{align}
Suppose that $f(x)= x+ \lambda \left( \log n - \log x
\right)^{\delta}$  reaches the minimum at $x^{\ast}$ when $1\leq
x\leq y$. It is easy to check that $f'(x)=1-\lambda \delta (\log n -
\log x)^{\delta-1}/x$ is monotonically increasing for $x$ in $(0,
n)$. Then, by defining
$$ x_1 \eqdef \lambda
\delta(\log n)^{\delta-1} - (1-\epsilon)\lambda
\delta(\delta-1)^2(\log \log n) (\log n)^{\delta-2},\;\;\epsilon>0,
$$
we obtain, after some easy calculations, for large $n$,
$$
 f'(x_1) \geq 1-\frac{(\log n)^{\delta-1}-(1-\epsilon/2)(\delta-1)^2(\log \log n) (\log n)^{\delta-2}}
 {(\log n)^{\delta-1} - (1-\epsilon)(\delta-1)^2(\log \log n) (\log n)^{\delta-2}}> 0,
$$
which implies that $f'(x)>0$ for $x\geq x_1$ and, therefore,
$x^{\ast}<x_1$ for all $n>n_0$. Hence, by (\ref{eq:gtrOneupper1}),
we obtain
\begin{align}
  \Pr[N>n]  & \leq  (1+\epsilon)y e^{1 - \lambda \left( \log n - \log
  x_1
  \right)^{\delta}}+e^{-y}+
  o\left( \Pr[N>n] \right),
\end{align}
which, by recalling the definitions of $y$ and $x_1$,  results in
\begin{equation}\label{eq:gtrOneupperA}
 \log \Pr[N>n]^{-1} - \lambda(\log n)^{\delta} \gtrsim  -(1+\epsilon)\lambda \delta(\delta-1) (\log \log n)(\log n)^{\delta-1}.
\end{equation}
Finally, passing $\epsilon \to 0$ in (\ref{eq:gtrOneupperA}) and
combining it with (\ref{eq:gtrOnelowerA}), we finish the proof. \qed
\end{proof}

\subsection{Proof of Proposition \ref{prop:betweenHalfOne}}\label{ss:betweenHalfOne}
\begin{proof}
First, we prove the \emph{lower bound}. Using the same arguments as
in the proof of the lower bound for Theorem~\ref{theorem:2I}, we
obtain, for $0<\epsilon<1, x_0>0$ and $n$ large enough,
\begin{align}
 \log \left(\Pr[N>n]^{-1} \right) &\leq  x_0 +\frac{1}{1-\epsilon} \log \left(\Phi\left(\frac{(1+\epsilon)n}{x_0}
  \right)\right)=x_0+\frac{1}{1-\epsilon}e^{\lambda \left( \log \left( \frac{(1+\epsilon)n}{x_0} \right)
  \right)^{\delta}}. \nonumber
\end{align}
Setting $x_0=e^{\lambda (\log n)^{\delta} \left( 1- \delta \lambda
(\log n)^{\delta -1} \right) }, 1/2<\delta<1$ in the preceding
inequality yields
\begin{align}
 \log \left(\Pr[N>n]^{-1}\right)&\leq e^{\lambda (\log n)^{\delta} \left( 1- \delta \lambda
(\log n)^{\delta -1} \right) } + \frac{1}{1-\epsilon}e^{\lambda(\log
n - \log x_0 +\log (1+\epsilon))^{\delta}}\nonumber,
\end{align}
which, by noting that $\lambda(\log n - \log x_0 + \log
(1+\epsilon))^{\delta} \leq \lambda(\log n)^{\delta} \left( 1-
(1-\epsilon)\delta \lambda (\log n)^{\delta -1} \right)$ for all $n$
large enough, implies, for $n$ large enough,
\begin{equation}
  \log (\log \Pr[N>n]^{-1}) \leq  \log\left(1+ \frac{1}{1-\epsilon}\right)+ \lambda(\log
n)^{\delta} \left( 1- (1-\epsilon)\delta \lambda (\log n)^{\delta
-1} \right).\nonumber
\end{equation}
Passing $\epsilon \to 0$ in the preceding inequality results in
\begin{equation}\label{eq:propHalfOnelower}
  \log (\log \Pr[N>n]^{-1}) - \lambda (\log n)^{\delta} \leq  - \delta \lambda^2 (\log
  n)^{2\delta-1}.
\end{equation}

Next, we prove the \emph{upper bound}. Following the same approach
as in the proof of Theorem~\ref{theorem:1I}, we obtain
\begin{align}\label{eq:propHalfOneupper1}
  \Pr[N>n]  & \leq  \sum_{k=0}^{ y-1 } e^{-k- \frac{1}{1+\epsilon}e^{\lambda \left( \log n - \log k \right)^{\delta}}}
 + e^{-y}+
  o\left( \Pr[N>n] \right).
\end{align}
Choose $y=e^{\lambda (\log n)^{\delta} \left( 1- (1+\epsilon)\delta
\lambda (\log n)^{\delta -1} \right) }$ and let $f(x)= x+ e^{\lambda
\left( \log n - \log x \right)^{\delta}}/(1+\epsilon)$. Since
$f'(x)=1-e^{\lambda (\log  n - \log x)^{\delta}}/((1+\epsilon)x)$ is
an increasing function for $x$ in $(0,n)$, it is easy to see that,
for all $0<x \leq y$ and $n$ large enough,
$$
  f'(x) \leq 1-\frac{e^{\lambda (\log  n - \log y)^{\delta}}}{(1+\epsilon)y}
  \leq 1- \frac{e^{\lambda (\log n)^{\delta} \left( 1- \delta
\lambda (\log n)^{\delta -1} \right) }}
  {(1+\epsilon)e^{\lambda (\log n)^{\delta} \left( 1- (1+\epsilon)\delta
\lambda (\log n)^{\delta -1} \right) }}<0.
$$
Therefore, for $0\leq k\leq y$, we obtain
$$
  e^{-k- \frac{1}{1+\epsilon}e^{\lambda \left( \log n - \log k \right)^{\delta}}}\leq
   e^{-y-\frac{1}{1+\epsilon} e^{\lambda \left( \log n - \log y \right)^{\delta}}},
$$
which, by (\ref{eq:propHalfOneupper1}), yields
\begin{align}\label{eq:propHalfOneupper2}
  \Pr[N>n]  & \leq  y e^{-y- \frac{1}{1+\epsilon}e^{\lambda \left( \log n - \log y \right)^{\delta}}}
 + e^{-y}+
  o\left( \Pr[N>n] \right) \nonumber\\
  &\leq (y+1) e^{-y}+
  o\left( \Pr[N>n] \right),
\end{align}
implying
\begin{equation}\label{eq:propHalfOneupper}
  \log (\log \Pr[N>n]^{-1}) - \lambda (\log n)^{\delta} \gtrsim   -(1+\epsilon) \delta \lambda^2 (\log
  n)^{2\delta-1}.
\end{equation}
Finally, by passing $\epsilon \to 0$ in (\ref{eq:propHalfOneupper})
and combining it with (\ref{eq:propHalfOnelower}), we finish the
proof. \qed
\end{proof}

\subsection{Proof of Theorem \ref{theorem:asympT}}\label{ss:asympT}
 The proofs are based on large
deviation results developed by S. V. Nagaev in \cite{Na79};
specifically, we summarize Corollary 1.6 and Corollary 1.8 of
\cite{Na79} in this following lemma.
 \begin{lemma}
   \label{lemma:nagaev}
   Let $X_1, X_2,\cdots X_n$ and $X$ be i.i.d random variables with
   $\int_{u\geq 0}u^sd\Pr[X<u] < \infty$ and $\expect X=0$.

  If $1 \leq s \leq 2$, then, there exist finite $y_s, c>0$ such that for
  $x>y>y_s$,
 \begin{equation}\label{eq:nagaev1}
    \Pr\left[\sum_{i=1}^{n}X_i \geq x \right] \leq
 n\Pr[X>y]+\left( \frac{c n}{xy^{s-1}} \right)^{x/2y}.
 \end{equation}

  If $s>2$, then, there exist finite $c>0$ such that
        \begin{equation}\label{eq:nagaev2}
           \Pr\left[\sum_{i=1}^{n}X_i \geq x \right] \leq \frac{c n}{x^{s}} + \exp\left( \frac{- x^2}
            {cn}\right).
            \end{equation}
    \end{lemma}
\begin{proof}
  Please refer to \cite{Na79}.
\end{proof}

Now, we are ready to prove Theorem \ref{theorem:asympT}.

\begin{proof}
First, we establish the \emph{upper bound}. By recalling Definition
\ref{def:NT}, for any $1/2>\delta>0$, we obtain
\begin{align} \label{eq:upperT123}
  \Pr[T>t]&= \Pr\left[\sum_{i=1}^{N-1}(U_i+A_i)+L >t\right] \nonumber\\
  & \leq \Pr\left[ \sum_{i=1}^{N}(A_i \wedge L+ \expect[U]) > (1-2\delta)t
  \right]+ \Pr\left[ \sum_{i=1}^{N}(U_i- \expect[U]) > \delta t \right]
  +\Pr\left[L > \delta t\right] \nonumber\\
&\eqdef I_1+I_2+I_3.
\end{align}

The condition $\expect [L^{\alpha+\epsilon} ]<\infty$  implies
\begin{equation}\label{eq:upperT3}
I_3\leq \frac{\expect[L^{\alpha+\theta}]}{(\epsilon
t)^{\alpha+\theta}}= O\left( \frac{1}{t^{\alpha+\theta}}\right).
\end{equation}

For $I_2$,  we begin with studying the case of $\alpha>1$, i.e.,
when $\expect[N]<\infty$. Since $N$ is independent of $\{ U_i\}$, by
defining $X_i \eqdef U_i-\expect[U_i]$, we obtain,
$$
I_2 = \sum_{n=1}^{\infty} \Pr[N=n] \Pr\left[\sum_{i=1}^{n} X_i
> \delta t \right].
$$
To evaluate $\Pr\left[\sum_{i=1}^{n} X_i
> \delta t \right]$ in the preceding equality,  we need to apply Lemma \ref{lemma:nagaev}, which results in
 two situations. If $1<s \eqdef \alpha+\theta \leq 2$, using
(\ref{eq:nagaev1}) with $y=\delta t/2$, we obtain, for all $n\geq
1$,
\begin{align}\label{eq:forI22}
\Pr\left[\sum_{i=1}^{n} X_i>\delta
 t\right] &\leq n \Pr[X_1 >\delta t/2]+ \frac{
 2^{s-1}c n}{\delta^{s}{t}^{s}},
\end{align}
implying
\begin{align}\label{eq:I22}
I_2 &\leq \sum_{n=1}^{\infty} \Pr[N = n] \left(n \Pr[X_1 >\delta
t/2]+ \frac{
 2^{s-1}c n}{\delta^{s}{t}^{s}} \right) \nonumber\\
 &\leq \expect[N] \Pr[X_1 >\delta t/2]+ \frac{
 2^{s-1} c \expect[N]}{\delta^{s}{t}^{\alpha+\theta}} = O\left( \frac{1}{t^{\alpha+\theta}} \right).
\end{align}
 Otherwise, if $s=\alpha+\theta > 2$,  by
(\ref{eq:nagaev2}), we derive, for $0<\delta<1$,
$0<\gamma<\alpha\delta/(1+\delta)$,
\begin{align}
I_2 &\leq  \Pr\left[ \sum_{i=1}^{\lfloor t^{1+\delta} \rfloor}X_i >
\delta t
\right]+ \Pr\left[N > t^{1+\delta}\right]\nonumber\\
&=\sum_{n=1}^{\lfloor t^{1+\delta} \rfloor} \Pr[N=n]
\Pr\left[\sum_{i=1}^{n} X_i
> \delta t \right]+ O\left( \frac{1}{t^{(1+\delta)(\alpha-\gamma)}}
\right) \nonumber\\
&\leq \frac{c\expect[N]}{(\delta t)^{\alpha + \theta}} + \exp \left(
-\frac{\delta ^2 t^{1-\delta}}{c} \right)  + O\left(
\frac{1}{t^{(1+\delta)(\alpha-\gamma)}} \right), \nonumber
\end{align}
which implies, for some $\nu>0$,
\begin{equation}\label{eq:I2end}
I_2=O\left( \frac{1}{t^{\alpha+\nu}}  \right).
\end{equation}

Now, we study the case when $0<\alpha \leq 1$. For $1<s\eqdef
1+\theta \leq 2, \theta>0$, recalling (\ref{eq:forI22}) and noting
that $\sum_{n=1}^{\lfloor t^{\zeta} \rfloor}n\Pr[N=n]\leq H
t^{\zeta(1-\alpha+\sigma)}$ for $\alpha>\theta>\sigma>0,
(\theta+1)/(\sigma+1)>\zeta>1$, $H>0$, we obtain, for some $\nu>0$,
\begin{align}
I_2 &\leq \sum_{n=1}^{\lfloor t^{\zeta} \rfloor } \Pr[N = n] \left(n
\Pr[X_1
>\delta t/2]+ \frac{
 2^{s-1}c n }{\delta^{s}{t}^{s}} \right) + \Pr\left[N>t^{\zeta}\right] \nonumber\\
 &\leq H t^{\zeta(1-\alpha+\sigma)}\left( \frac{\expect\left[ X_1^{1+\theta}\right]}{(\delta t/2)^{1+\theta}}+ \frac{
 2^{s-1} c}{\delta^{s}{t}^{1+\theta}} \right) + \Pr\left[N> t^{\zeta}\right] \nonumber\\
 &= O\left( \frac{1}{t^{\alpha+\nu}} \right), \nonumber
\end{align}
which, in conjunction with (\ref{eq:I22}) and (\ref{eq:I2end}),
yields, for some $\nu>0$,
\begin{equation}\label{eq:I2endAll}
I_2=O\left( \frac{1}{t^{\alpha+\nu}}  \right).
\end{equation}

%

Next, we study $I_1$. It is easy to obtain, for $\epsilon>0$,
\begin{align}\label{eq:I1123}
I_1&\leq \Pr\left[ \sum_{i=1}^{\frac{(1-2\delta)t}{
\expect[A+U](1+\delta)}}(A_i\wedge (\epsilon t) + \expect[U])
> (1-\delta)t
  \right]  + \Pr\left[N > \frac{(1-2\delta)t}{
\expect[A+U](1+\delta)} \right]+\Pr[L>\epsilon t]\nonumber \\
&\eqdef I_{11}+I_{12}+I_{13}.
\end{align}
By recalling Theorem \ref{theorem:asympN2}, we know
\begin{equation}\label{eq:I12}
\Pr\left[N > \frac{(1-2\delta)t}{ \expect[A+U](1+\delta)} \right]
\thicksim \frac{\Gamma( \alpha+1 )(\expect[U+A](1+\delta)
)^{\alpha}}{ \Phi((1-2\delta)t) }.
\end{equation}
The same argument for (\ref{eq:upperT3}) implies
\begin{equation}\label{eq:I13}
I_{13} =O\left( \frac{1}{t^{\alpha+\theta}}\right).
\end{equation}
Furthermore, $I_{11}$ is upper bounded by
 \begin{align}
&\Pr\left[ \sum_{i=1}^{\frac{(1-2\delta)t}{
\expect[A+U](1+\delta)}}(A_i\wedge (\epsilon t) +
\expect[U])-(1+\delta)\expect[A+U] \frac{(1-2\delta)t}{\expect[A+U]
(1+\delta)}> \delta t
\right]\nonumber\\
&\;\;\;\;\;  \leq \Pr\left[\sup_{n} \left\{\sum_{i=1}^{n}(A_i \wedge
(\epsilon t) + \expect[U])-n(1+\delta)\expect[A+U] \right\}
> \delta t \right],\nonumber
 \end{align}
where in the preceding probability,   $\sup_{n}
\left\{\sum_{i=1}^{n}(A_i \wedge (\epsilon t) +
\expect[U])-n(1+\delta)\expect[A+U] \right\}$ is equal in
distribution to the stationary workload in a $D/GI/1$ queue with
truncated service times with the stability condition $\expect
[(A\wedge (\epsilon t) +\expect[U])]< (1+\delta)\expect[A+U]$.
Therefore, using a similar proof for Lemma 3.2 of \cite{JEM01b}, we
can show that for any $\beta>0$, there exists $\epsilon>0$ such that
$$
I_{11}=o\left( \frac{1}{t^{\beta}} \right),
$$
 which, in
conjunction with (\ref{eq:I12}), (\ref{eq:I13}), (\ref{eq:I1123}),
 and (\ref{eq:upperT123}),
(\ref{eq:upperT3}), (\ref{eq:I2endAll}), yields,  by passing
$\epsilon, \delta \to 0$ in (\ref{eq:I12}),
\begin{equation}\label{eq:upperT}
\Pr\left[T > t \right] \lesssim \frac{\Gamma( \alpha+1
)(\expect[U+A])^{\alpha}}{\Phi(t) }.
\end{equation}

Then, we prove the \emph{lower bound}. It is easy to obtain, for
$\delta>0$,
\begin{align} \label{eq:lowerT1}
  \Pr[T>t] &= \Pr\left[\sum_{i=1}^{N-1}(U_i+A_i)+L>t\right] \nonumber\\
  &\geq \Pr\left[\sum_{i=1}^{N-1}(U_i+A_i)>t, N\geq \frac{t(1+\delta)}{\expect[U+A]}+1\right]\nonumber\\
  &\geq \Pr\left[N\geq \frac{t(1+\delta)}{\expect[U+A]}+1\right]
      -\Pr\left[\sum_{i=1}^{N-1}(U_i+A_i)\leq t, N\geq
   \frac{t(1+\delta)}{\expect[U+A]}+1\right] \nonumber\\
   &\eqdef I_1-I_2.
\end{align}
For $I_2$, by defining $Y_i\eqdef U_i+A_i -\expect[U+A]$, we obtain
\begin{align*}
 I_2 &\leq \Pr\left[\sum_{i \leq t(1+\delta)/\expect[U+A]} (U_i+A_i)\leq t
 \right] =\Pr\left[\sum_{i \leq t(1+\delta)/\expect[U+A]} (-Y_i)\geq \delta
 t\right] \nonumber
\end{align*}
with $(-Y_i) \leq \expect[U+A] <\infty$.
By Chernoff bound, there exists $h, \eta>0$, such that
\begin{equation}\label{eq:lowerT2}
I_2\leq O\left( he^{-\eta t} \right),
\end{equation}
which,  by Theorem \ref{theorem:asympN2}, equation
(\ref{eq:lowerT1}) and passing $\delta \to 0$, yields
\begin{equation}\label{eq:lowerT}
\Pr\left[T > t \right] \gtrsim \frac{\Gamma( \alpha+1
)(\expect[U+A])^{\alpha}}{\Phi(t) }.
\end{equation}
Combining (\ref{eq:upperT}) and (\ref{eq:lowerT}) completes the
proof.  \qed
\end{proof}

\subsection{Proof of Theorem \ref{theorem:T}}\label{ss:ProofSlowT}
\begin{proof}
First, we prove the \emph{upper bound}. It is easy to see that for
$0<\epsilon<1$,
\begin{align}\label{eq:thrmUpI123}
 \Pr[T>(1+\epsilon)t]&=\Pr\left[\sum_{i=1}^{N-1}((A_i \wedge L) + U_i) + L > (1+\epsilon)t
 \right]\nonumber \\
       &\leq \Pr\left[ \sum_{i=1}^{\lceil t/l(t) \rceil}(A_i \wedge L) > \frac{t}{2} \right]
   +\Pr\left[ \sum_{i=1}^{\lceil t/l(t) \rceil}U_i> \frac{t}{2} \right] +
 \Pr\left[N> {\Big \lceil} \frac{t}{l(t)} {\Big \rceil}+1 \right]\nonumber\\
 &\;\;\;\;  + \Pr[L> \epsilon t] \nonumber \\
 & \eqdef I_1 + I_2 + I_3+I_4.
\end{align}

Now, since $l(\cdot)$ is slowly varying and
$\Pr[L>x]=O(\Phi(x)^{-1})$, we obtain,
\begin{equation}\label{eq:thrmTI4}
I_4=\Pr[L>t]=o\left(\Phi(t)^{1-\epsilon}\right).
\end{equation}
By Theorem \ref{theorem:1I}, we obtain
\begin{equation}
  \lim_{t\to \infty} \frac{\log \left( \Pr\left[N> {\Big \lceil} \frac{t}{l(t)} {\Big \rceil} +1 \right]^{-1} \right)}
               {\log \Phi\left( \frac{t}{l( t)}
  \right)}=1, \nonumber
\end{equation}
which, by (\ref{eq:lx}), yields
\begin{equation}\label{eq:thrmTI3}
  \lim_{t\to \infty} \frac{\log \left( \Pr\left[N> {\Big \lceil} \frac{t}{l(t)}
          {\Big \rceil}+1\right]^{-1} \right)}{\log \left(\Phi (t)
  \right)}=1.
\end{equation}

Next, we evaluate $I_1$ and $I_2$.  For $I_2$,
\begin{align}\label{eq:thrmTI2}
I_2&= \Pr\left[ \sum_{i=1}^{\lceil t/l(t) \rceil}U_i> \frac{t}{2}
\right]\nonumber\\
&\leq \left\lceil \frac{t}{l(t)} \right\rceil \Pr\left[ U_1>
\frac{t}{l(t)} \right]+\Pr\left[ \sum_{i=1}^{\lceil t/l(t)
\rceil}U_i \wedge
\frac{t}{l(t)}> \frac{t}{2} \right]\nonumber\\
&\eqdef I_{21}+I_{22}.
\end{align}
For $\delta>0$ and large $t$,  due to condition (\ref{eq:lxB}) we
obtain $l(t/l(t))\geq (1-\delta/2) l(t)$,  which yields
\begin{align}\label{eq:thrmTI21}
I_{21}&\leq O\left( e^{-(1+\delta)l\left( \frac{t}{l(t)}
\right)}\right) \leq O\left(
e^{-(1+\delta)(1-\frac{\delta}{2})l\left(t \right)}\right) = o\left(
\Phi(t)^{-1} \right).
\end{align}
Then, by using Chernoff bound, for $h>0$, we obtain
\begin{align}
I_{22}&= \Pr\left[ e^{ h \left(\sum_{i=1}^{\lceil t/l(t) \rceil}X_i
\wedge
\frac{t}{l(t)} \right) }> e^{ht/2} \right]\nonumber\\
&\leq e^{-\frac{ht}{2}}\left( \expect\left[e^{ h \left(X_i \wedge
\frac{t}{l(t)} \right) } \right] \right)^{\frac{t}{l(t)}+1},
\nonumber
\end{align}
which, by selecting $h=4l(t)/t$, and noting that
$$e^{ h \left(X_i
\wedge \frac{t}{l(t)} \right) } \leq 1+ (e^{4}-1) \frac{l(t)}{t}
\left(X_1 \wedge \frac{t}{l(t)} \right), $$
 implies
\begin{align}\label{eq:thrmTI22}
I_{22}&\leq e^{-2l(t)}\left( \expect\left[1+ (e^{4}-1)
\frac{l(t)}{t} \left(X_1 \wedge \frac{t}{l(t)} \right)
\right]\right)^{\frac{t}{l(t)}+1} \nonumber\\
&\leq e^{-2l(t)}\left( 1+ (e^{4}-1) \frac{l(t)}{t} \expect\left[X_1
 \right]\right)^{\frac{t}{l(t)}+1} \nonumber\\
 &=o\left( \frac{1}{\Phi(t)} \right).
\end{align}
Combining (\ref{eq:thrmTI2}), (\ref{eq:thrmTI21}) and
(\ref{eq:thrmTI22}) yields
\begin{equation}\label{eq:thrmTI2r}
I_2= o\left(\frac{1}{\Phi(t)} \right).
\end{equation}
For $I_1$, it is easy to see
\begin{align}
I_1&= \Pr\left[ \sum_{i=1}^{\lceil t/l(t) \rceil}(A_i \wedge L)>
\frac{t}{2}
\right]\nonumber\\
&\leq \Pr\left[\sum_{i=1}^{\lceil t/l(t) \rceil} \left(A_i \wedge
\frac{t}{l(t)} \right)>
\frac{t}{l(t)} \right]+\Pr\left[ L>\frac{t}{l(t)}\right]\nonumber\\
&\eqdef I_{11}+I_{12}.
\end{align}
Using the same argument as in deriving (\ref{eq:thrmTI22}), we can
prove that $I_{11}=o(1/\Phi(t))$, which, by noting condition
(\ref{eq:lxB}) implying $I_{12}=o(1/\Phi(t))$, yields
\begin{equation}\label{eq:thrmTI1r}
I_1= o\left(\frac{1}{\Phi(t)} \right).
\end{equation}
Combining (\ref{eq:thrmUpI123}), (\ref{eq:thrmTI4}),
(\ref{eq:thrmTI3}), (\ref{eq:thrmTI2r}) and (\ref{eq:thrmTI1r}),
yields
\begin{equation}\label{eq:thrmUp1}
 \varlimsup_{t \rightarrow \infty}\frac{\log \Pr[T>t]}{ \log(\Phi
   (t))} \leq -1.
\end{equation}

Next, we prove the \emph{lower bound}. Observe
\begin{align}
 \Pr\left[\sum_{i=1}^{N-1}(A_i + U_i) + L >t \right]
       &\geq \Pr\left[ \sum_{i=0}^{N-1}(A_i \wedge 1)   >t ,
       N >  {\Big \lceil} \frac{2t}{\expect[A\wedge 1]} {\Big \rceil} +1
       \right]\nonumber \\
       &\geq  \Pr\left[ N >  {\Big \lceil} \frac{2t}{\expect[A\wedge 1]} {\Big \rceil} +1\right]
           -\Pr\left[ \sum_{i=1}^{  {\Big \lceil} \frac{2t}{\expect[A\wedge 1]} {\Big \rceil} }(A_i \wedge 1) \leq t \right],
           \nonumber
\end{align}
and, by using the same arguments as in deriving (\ref{eq:lowerT2}),
it is very easy to prove that the second probability on the right
hand side of the second
 inequality above is exponentially bounded.  Therefore, using Theorem \ref{theorem:1I} and the
 preceding exponential bound yields
\begin{equation}\label{eq:thrmLower1}
 \varliminf_{t \rightarrow \infty}\frac{\log \Pr[T>t]}{ \Phi(\log
   t)} \geq -1.
\end{equation}
Combining (\ref{eq:thrmUp1}) and (\ref{eq:thrmLower1}) completes the
proof. \qed
\end{proof}

\subsection{Proof of Theorem \ref{theorem:TWeibull}}\label{ss:ProofWeibull}
\begin{proof}
 First, we prove the upper
bound. It is easy to see that, for $\eta\eqdef
\expect[U]/\expect[A+U]$ and $0<\epsilon<1$,
\begin{align}\label{eq:WeibullUpI123}
 \Pr[T>(1+\epsilon)t]&=\Pr\left[\sum_{i=1}^{N-1}((A_i \wedge L) + U_i) + L >t
 \right]\nonumber \\
       &\leq \Pr\left[ \sum_{i=1}^{\left\lfloor \frac{(1-\epsilon)t}{\expect[A+U]} \right\rfloor}(A_i\wedge L)
          > (1-\eta) t \right] +
\Pr\left[ \sum_{i=1}^{\left\lfloor
\frac{(1-\epsilon)t}{\expect[A+U]}
\right\rfloor}U_i> \eta t \right]\nonumber\\
&\;\;\; + \Pr\left[N> \left\lfloor \frac{(1-\epsilon)t}{\expect[A+U]} \right\rfloor  \right] + \Pr[L> \epsilon t] \nonumber \\
 & \eqdef I_1 + I_2 + I_3 +I_4.
\end{align}

The condition on $L$ implies
\begin{equation}\label{eq:WeibullTI4}
I_4=\Pr[L> \epsilon t]=o\left( e^{-(\log \Phi(x))^{1/(\beta+1)}}
\right),
\end{equation}
and, by Theorem \ref{theorem:2I}, we obtain
\begin{equation}\label{eq:WeibullTI3}
   \lim_{t \rightarrow \infty}\frac{ \log
         \left(\Pr\left[N > \left\lfloor \frac{(1-\epsilon)t}{\expect[X_1]} \right\rfloor \right]^{-1} \right)
         }{ \left(\log\Phi(
   t) \right)^{\frac{1}{\beta+1}}}= (1-\epsilon)^{\frac{\beta}{\beta+1}}\frac{ \beta^{\frac{1}{\beta+1}} +
\beta^{-\frac{\beta}{\beta+1}}
}{(\expect[A+U])^{\frac{\beta}{\beta+1}}}.
\end{equation}

Now, we evaluate $I_2$.  By applying the large deviation result
proved in Theorem 3.2 (ii) of \cite{JEM01b}, and noting
$\Pr[U>x]\leq o\left( e^{-x^{(1+\delta/2)\beta/(\beta+1)}} \right)$,
we can prove that there exist $1>\gamma>0$ and $C>0$, such that
\begin{align}\label{eq:WeibullU}
\Pr\left[ \sum_{i=1}^{\left\lfloor
\frac{(1-\epsilon)t}{\expect[A+U]} \right\rfloor} (U_i \wedge \gamma
\epsilon \eta t) - \eta (1-\epsilon)t > \epsilon \eta t \right]&\leq
C \left( e^{-(\epsilon \eta
t)^{(1+\delta/2))\beta/(\beta+1)}} \right)\nonumber\\
&=o\left( e^{-(\log \Phi(x))^{1/(\beta+1)}} \right).
\end{align}
Thus, considering $I_2$, we obtain
\begin{align}
I_2&\leq  \left\lfloor \frac{(1-\epsilon)t}{\expect[X_1]}
\right\rfloor \Pr\left[  U_1 > (\gamma \epsilon \eta) t  \right]+
\Pr\left[ \sum_{i=1}^{\left\lfloor
\frac{(1-\epsilon)t}{\expect[X_1]} \right\rfloor} \left( U_i \wedge
\gamma \epsilon \eta t \right) > \eta t \right],
\end{align}
which, by (\ref{eq:WeibullU}) and the assumption on $U$, yields,
\begin{align}\label{eq:WeibullTI2}
I_2&=  o\left( e^{-(\log \Phi(x))^{1/(\beta+1)}} \right).
\end{align}

For $I_1$, we begin with proving the situation when $\zeta=0$,
$\xi>\beta$, i.e.,  assuming no conditions on $\Pr[A>x]$ beyond
$\expect[A]<\infty$. It is easy to obtain, for $0<\epsilon
<1/(\beta+1)$,
\begin{align}\label{eq:WeibullTI1}
I_1&= \Pr\left[ \sum_{i=1}^{\left\lfloor
\frac{(1-\epsilon)t}{\expect[A+U]} \right\rfloor}(A_i \wedge L)>
(1-\eta)t
\right]\nonumber\\
&\leq \Pr\left[ L> t^{\frac{1}{\beta+1}-\epsilon} \right]+\Pr\left[
\sum_{i=1}^{\left\lfloor \frac{(1-\epsilon)t}{\expect[A+U]}
\right\rfloor} \left( A_i \wedge
t^{\frac{1}{\beta+1}-\epsilon} \right) > (1-\eta)t \right]\nonumber\\
&\eqdef I_{11}+I_{12}.
\end{align}
The condition $\xi>\beta$ implies, for $0<\epsilon <
(1-\beta/\xi)/(\beta+1)$,
\begin{align}\label{eq:WeibullTI11}
I_{11}&\leq O\left( e^{ -t^{\left(\frac{1}{\beta+1}-\epsilon
\right)\xi} }\right)  = o\left( e^{ -(\log \Phi(t))^{1/(\beta+1)}}
\right).
\end{align}
And, by using Chernoff bound, for $h>0$, we obtain
\begin{align}
I_{12}&= \Pr\left[ e^{ h \left( \sum_{i=1}^{\left\lfloor
\frac{(1-\epsilon)t}{\expect[A_1]} \right\rfloor} \left( A_i \wedge
t^{\frac{1}{\beta+1}-\epsilon} \right) \right) }> e^{h(1-\eta)t}
\right] \leq e^{-h(1-\eta)t}\left( \expect\left[e^{ h \left( A_1
\wedge t^{\left(\frac{1}{\beta+1}-\epsilon\right)} \right) } \right]
\right)^{\left\lfloor \frac{(1-\epsilon)t}{\expect[A_1]}
\right\rfloor}, \nonumber
\end{align}
which, by selecting $h=\epsilon (1-\eta)
t^{-\left(\frac{1}{\beta+1}-\epsilon \right)}$, and using $e^x\leq
1+(e^{b}-1)x/b$ for $0\leq x \leq b$, yields
$$e^{ h \left(A_1
\wedge t^{\left(\frac{1}{\beta+1}-\epsilon \right)} \right) } \leq
1+ \frac{e^{\epsilon(1-\eta)}-1}{\epsilon(1-\eta)} h \left(A_1
\wedge t^{\left(\frac{1}{\beta+1}-\epsilon \right)} \right).
$$
Then, the preceding inequalities, for $\epsilon$ small enough such
that $\epsilon(1-\eta)- (1-\epsilon)(e^{\epsilon(1-\eta)}-1)>0$,
imply
\begin{align}\label{eq:WeibullTI12}
I_{12}&\leq e^{-\epsilon (1-\eta)
t^{\frac{\beta}{\beta+1}+\epsilon}}\left( \expect\left[1+
\frac{e^{\epsilon(1-\eta)}-1}{\epsilon(1-\eta)} h \left(A_1 \wedge
t^{\left(\frac{1}{\beta+1}-\epsilon \right)} \right)
\right]\right)^{\left\lfloor \frac{(1-\epsilon)t}{\expect[A_1]}
\right\rfloor} \nonumber\\
&\leq e^{-\epsilon(1-\eta) t^{\frac{\beta}{\beta+1}+\epsilon}}\left(
1+ (e^{\epsilon(1-\eta)}-1)
\frac{t^{\frac{\beta}{\beta+1}+\epsilon}}{t} \expect\left[A_1
 \right]\right)^{\left\lfloor \frac{(1-\epsilon)t}{\expect[A_1]}
\right\rfloor} \nonumber\\
 &=O\left( e^{-\left(\epsilon(1-\eta)- (1-\epsilon)(e^{\epsilon(1-\eta)}-1) \right)
 t^{\frac{\beta}{\beta+1}+\epsilon}}\right)= o\left( e^{ -(\log
\Phi(t))^{1/(\beta+1)}} \right).
\end{align}
Combining (\ref{eq:WeibullTI11}) and (\ref{eq:WeibullTI12}) yields
$I_1=o\left( e^{ -(\log \Phi(t))^{1/(\beta+1)}} \right)$ for
$\zeta=0, \xi>\beta$.

Now, in order to  prove the situation $\zeta>0$ when $\Pr[A>x]$ is
bounded by a Weibull distribution, we need to use the following
lemma that is based on a minor modification of Theorem 3.2 (ii) in
\cite{JEM01b} (or Lemma 2 in \cite{JEM03a}) that can be proved by
selecting $s=\upsilon Q(u)/u, 0<\upsilon<1$ in (5.18) of
\cite{JEM01b}, where $Q(u)$ is defined in \cite{JEM01b}.
\begin{lemma}\label{thrm:jelenkovic2}
If $\Pr[A>x]\leq He^{-x^{\zeta}}$, $H>0, 1>\zeta>0$, then, for $
x^{\theta} < u < \epsilon x$, $\epsilon>0$, $1>\theta>0$ and $ n
\leq H x$, there exist $C>0, 1>\delta>0$, such that
\begin{equation}
\Pr\left[ \sum_{i=1}^{n}A_i \wedge u - n \expect[A]>x \right]\leq C
e^{-\delta u^{\zeta-1}x}. \nonumber
\end{equation}
\end{lemma}
Note that the case $\zeta\geq 1$ is trivial since in this situation
$I_1$ is exponentially bounded using Chernoff bound. Therefore, we
only need to consider the situation $0<\zeta<1$.  Using the union
bound and the independence of $\{A_i\}$ and $L$, it is easy to
obtain, for $0<\epsilon<1/(\beta+1)$,
\begin{align}\label{eq:WeibullTI1B}
I_1&= \Pr\left[ \sum_{i=1}^{\left\lfloor
\frac{(1-\epsilon)t}{\expect[A+U]} \right\rfloor}(A_i \wedge L)>
(1-\eta) t
\right]\nonumber\\
&\leq \Pr\left[ L> \epsilon t \right]+ \Pr\left[
\sum_{i=1}^{\left\lfloor \frac{(1-\epsilon)t}{\expect[A+U]}
\right\rfloor} \left( A_i \wedge t^{\frac{1}{\beta+1}-\epsilon}
\right) > (1-\eta)t \right]\nonumber\\
&\;\;\; + \int_{t^{\frac{1}{\beta+1}-\epsilon}}^{\epsilon t}
\Pr\left[ \sum_{i=1}^{\left\lfloor
\frac{(1-\epsilon)t}{\expect[A+U]} \right\rfloor}(A_i \wedge u)>
(1-\eta)t
\right]d\Pr[L\leq u]  \nonumber\\
&\eqdef I_{11}+I_{12}+I_{13}.
\end{align}

From (\ref{eq:WeibullTI4}) and (\ref{eq:WeibullTI12}), we obtain
\begin{equation}\label{eq:WeibullTIB1112}
  I_{11}+I_{12} = o\left( e^{ -(\log
\Phi(t))^{1/(\beta+1)}} \right).
\end{equation}

Applying Lemma \ref{thrm:jelenkovic2} yields, for
$t^{1/(\beta+1)-\epsilon}\leq u \leq \epsilon t$,
\begin{align}
\Pr\left[ \sum_{i=1}^{\left\lfloor
\frac{(1-\epsilon)t}{\expect[A+U]} \right\rfloor}(A_i \wedge u)>
(1-\eta)t \right]& = \Pr\left[ \sum_{i=1}^{\left\lfloor
\frac{(1-\epsilon)t}{\expect[A+U]} \right\rfloor} \left( A_i \wedge
u \right)-\left\lfloor \frac{(1-\epsilon)t}{\expect[A+U]}
\right\rfloor \expect[A] > \epsilon (1-\eta)t \right]\nonumber\\
&\leq C e^{-\delta \epsilon(1-\eta)u^{\zeta-1}t},\nonumber
\end{align}
resulting in, for some $h>0$,
\begin{align}
I_{13}&\leq \int_{t^{\frac{1}{\beta+1}-\epsilon}}^{\epsilon t} C
e^{-\delta \epsilon(1-\eta)t u^{\zeta-1}}d\Pr[L\leq u]\nonumber\\
&\leq C e^{-\delta \epsilon(1-\eta)t u^{\zeta-1}} \Pr[L> u] {\Big
|}^{t^{\frac{1}{\beta+1}-\epsilon}}_{\epsilon t} +
\int_{t^{\frac{1}{\beta+1}-\epsilon}}^{\epsilon t} He^{-u^{\xi}}C
e^{-\delta \epsilon(1-\eta)t u^{\zeta-1}} (1-\zeta)\delta \epsilon
(1-\eta) t
u^{\zeta-2}du \nonumber\\
&\leq \sup_{t^{\frac{1}{\beta+1}-\epsilon}\leq  u \leq \epsilon t
}\left\{C e^{-u^{\xi}-\delta \epsilon(1-\eta)t u^{\zeta-1}} \right\}
\left( 1+  \int_{t^{\frac{1}{\beta+1}-\epsilon}}^{\epsilon t} H
(1-\zeta)\delta \epsilon (1-\eta) t
u^{\zeta-2}du \right) \nonumber\\
&=O\left( e^{-ht^{\xi/(\xi+1-\zeta)}} \right)= o\left( e^{ -(\log
\Phi(t))^{1/(\beta+1)}} \right).\nonumber
\end{align}
The preceding bound on $I_{13}$, in conjunction with
(\ref{eq:WeibullTIB1112}) and the proof of the case for $\zeta=0$,
implies, for all $\zeta\geq 0$,
\begin{equation}\label{eq:WeibullTIB1}
I_{1}=o\left( e^{ -(\log \Phi(t))^{1/(\beta+1)}} \right).
\end{equation}
Thus,  combining  (\ref{eq:WeibullUpI123}),  (\ref{eq:WeibullTI4}),
(\ref{eq:WeibullTI3}), (\ref{eq:WeibullTI2}),
  (\ref{eq:WeibullTIB1}), and passing $\epsilon
\to 0$ yields
\begin{equation}\label{eq:WeibullUp1}
  \lim_{t \rightarrow \infty}\frac{ \log \Pr[T>t]^{-1}}{ \left(\log\Phi(
   t) \right)^{\frac{1}{\beta+1}}} \geq \frac{ \beta^{\frac{1}{\beta+1}} +
\beta^{-\frac{\beta}{\beta+1}}
}{(\expect[A+U])^{\frac{\beta}{\beta+1}}}.
\end{equation}

Now, we prove the \emph{lower bound}. Using the same argument as in
deriving equation (\ref{eq:lowerT1}) in the proof of the lower bound
for Theorem \ref{theorem:asympT}, it is easy to obtain, for
$\delta>0$,
\begin{align} 
  \Pr[T>t]  &\geq \Pr\left[N\geq \frac{t(1+\delta)}{\expect[U+A]}+1\right]
      -\Pr\left[\sum_{i=1}^{N-1}(U_i+A_i)\leq t, N\geq
   \frac{t(1+\delta)}{\expect[U+A]}+1\right], \nonumber
\end{align}
where the second probability on the right hand side of the preceding
inequality is exponentially bounded (see (\ref{eq:lowerT2})).
Therefore, using Theorem \ref{theorem:2I} and passing $\delta \to 0$
yields
\begin{equation}\label{eq:WeibullLower1}
  \lim_{t \rightarrow \infty}\frac{ \log \Pr[T>t]^{-1}}{ \left(\log\Phi(
   t) \right)^{\frac{1}{\beta+1}}} \leq \frac{ \beta^{\frac{1}{\beta+1}} +
\beta^{-\frac{\beta}{\beta+1}}
}{(\expect[A+U])^{\frac{\beta}{\beta+1}}}.
\end{equation}
Combining (\ref{eq:WeibullUp1}) and (\ref{eq:WeibullLower1})
completes the proof.  \qed
\end{proof}

\small
\bibliographystyle{abbrv}
\bibliography{tcJournal}
\end{document}